\documentclass[%
 jmp,
 amsmath,amssymb,
preprint,%
]{revtex4-1}

\usepackage{graphicx}
\usepackage{dcolumn}
\usepackage{bm}

\usepackage[utf8]{inputenc}
\usepackage[T1]{fontenc}
\usepackage{mathptmx}

\usepackage{amsthm}
\usepackage{amsfonts}
\usepackage{fancyhdr}
\usepackage{mathrsfs}
\usepackage{latexsym}
\usepackage[usenames]{color}
\usepackage{graphicx}
\usepackage{amsmath}
\usepackage{epstopdf}
\usepackage{enumitem}

\newtheorem{definition}{Definition}[section]
\newtheorem{lemma}{Lemma}[section]
\newtheorem{theorem}{Theorem}[section]
\newtheorem{proposition}{Proposition}[section]
\newtheorem{corollary}{Corollary}[section]
\newtheorem{remark}{Remark}[section]
\newtheorem{assumption}{Assumption}[section]

\newcommand{\ma}{\begin{pmatrix}}
\newcommand{\am}{\end{pmatrix}}

\def\cA{\mathcal{A}}

\newcommand{\tr}{\operatorname{Tr}}
\newcommand{\dd}{\operatorname{d}\!}
\newcommand{\Res}{\operatorname{Res}\,}
\newcommand{\sgn}{\operatorname{sgn}\,}

\newcommand{\be}[1]{\begin{equation} \label{#1}}
\newcommand{\ee}{\end{equation}}
\newcommand{\pdpd}[2]{\frac{\partial #1}{\partial #2}}
\newcommand{\dede}[2]{\frac{{\operatorname{d}} #1}{{\operatorname{d}} #2}}

\newcommand{\ii}{\mathrm{i}}
\newcommand{\R}{\mathbb{R}}
\newcommand{\C}{\mathbb{C}}

\renewcommand{\Im}{\operatorname{Im}}
\renewcommand{\Re}{\operatorname{Re}}

\def \vphi {\varphi}
\def \L {\Lambda}
\newcommand{\gM}{\mathfrak{M}}
\def\bF{{\bf F}}
\def\bM{{\bf M}}
\def\bND{{\bf ND}}
\def\bG{{\bf G}}

\def\bcG{{\boldsymbol{\mathfrak{G}}}}
\def\bS{{\bf S}}
\def\bT{{\bf T}}

\def\bB{{\bf B}}

\def\bD{{\bf D}}
\def\bm1{{\bf m}_1}
\def\bR{{\bf R}}
\def\cK{\mathcal{K}}
\def\cR{\mathcal{R}}

\def\bcA{{\boldsymbol{\mathfrak{A}}}}
\def\bcB{{\boldsymbol{\mathfrak{B}}}}
\def\bcC{{\boldsymbol{\mathfrak{C}}}}

\newcommand{\mstrut}[1]{\mbox{\rule{0mm}{#1}}}

\begin{document}


\title{Inverse problem for the Rayleigh system with spectral data}

\author{Maarten V. de Hoop}
 \affiliation{Simons Chair in Computational and Applied Mathematics and
    Earth Science, \\
    Rice University, Houston, TX, 77005, USA}
 \email{mdehoop@rice.edu}
\author{Alexei Iantchenko}
 \affiliation{Department of
    Materials Science and Applied Mathematics, Faculty of Technology
    and Society, \\
    Malm\"{o} University, SE-205 06 Malm\"{o}, Sweden}
 \email{ai@mau.se}



\date{\today}

\begin{abstract}
We analyze an inverse problem associated with the time-harmonic Rayleigh system on a flat elastic half-space concerning the recovery of Lam\'{e} parameters in a slab beneath a traction-free surface. We employ the Markushevich substitution, while the data are captured in a Jost function, and point out parallels with a corresponding problem for the Schr\"{o}dinger equation. The Jost function can be identified with spectral data. We derive a Gel'fand-Levitan type equation and obtain uniqueness with two distinct frequencies.
\end{abstract}


\maketitle

\section{Introduction} 

In this paper, we study an inverse problem associated with the
Rayleigh system on a flat elastic half-space concerning the conditional unique recovery of
the Lam\'{e} parameters in a slab beneath a traction-free
surface. In the process, we point out parallels with a corresponding problem for the Schr\"{o}dinger equation.
The analysis involves the study of spectral properties of a matrix
Sturm-Liouville operator on the half-line with a Robin-type boundary
condition associated with elastic surface waves of Rayleigh type. This
operator is non-self-adjoint and contains the spectral
parameter in the boundary condition. It originates from the standard
Rayleigh boundary value problem by means of the so-called Markushevich
substitution. In this boundary value problem, the Lam\'{e} parameters,
$\lambda$ and $\mu$, as well as the density of mass, $\rho$, appear as
material parameters and are assumed to be functions of the boundary normal coordinate or
``depth'', $Z$ say. In our analysis, however, $\rho$ will not play a
role. In previous work \cite{dHIVZ-Rayleigh}, we analyzed the inverse
problem associated with the Rayleigh system using the semiclassical
spectrum as the data. Essentially, this semiclassical inverse problem
allows the recovery of one of the Lam\'{e} parameters, namely the
shear modulus $\mu$.


Using also an adjoint Markushevich substitution, we develop a spectral theory which
inherits many features of self-adjoint matrix-valued Sturm-Liouville
problems as in \cite{Bondarenko2015}. This follows from the fact that
the original Rayleigh operator is self-adjoint. We construct a Gel'fand-Levitan equation. As an aside, the approach presented, here, allows a
generalization from a traction-free surface to an isotropic
solid-fluid boundary, leading to Scholte-Gogoladze waves,
assuming that the fluid is homogeneous and known. The Markushevich substitution
was introduced by Markushevich \cite{Markushevich1986,
  Markushevich1989, Markushevich1994} following
ideas of Pekeris \cite{Pekeris1934} and was recently revisited in
\cite{ArgatovIantchenko2019}.

The idea of applying the Markushevich substitution originates in the
work of Beals, Henkin and Novikova \cite{Bealsetal1995}, where a
spectral analysis was performed in the context of exponentially
increasing quantities, $\hat{\lambda}(Z) = \rho(Z)^{-1} \lambda(Z)$,
$\hat{\mu}(Z) = \rho(Z)^{-1} \mu(Z)$ as $Z \rightarrow -\infty$, with
$Z \in (-\infty,0]$, which differs considerably from the assumptions
  in \cite{dHINZ} where $\hat{\lambda}(Z)$, $\hat{\mu}(Z)$ are
  constant, with values $\hat{\lambda}_0$, $\hat{\mu}_0$,
  respectively, beneath a certain depth $Z = -H$ while
  \textit{preserving} the Rayleigh system as $Z \to -\infty$ and
  enabling application in seismology with the slab signifying Earth's crust. Our data are necessarily different from the data considered by Beals, Henkin and Novikova. We consider the recovery of
  $\hat{\lambda}$ and $\hat{\mu}$ in a \textit{slab} with known
  thickness, $H$, from the Jost function. We show that the Gel'fand-Levitan equation \cite{Marchenko1986} can be constructed in our case and has
  a unique solution.

In the inverse problem, boundary spectral data directly
encode the Weyl matrix of the transformed problem, and equivalently
the Neumann-to-Dirichlet map associated with the original Rayleigh
system. Boundary spectral data are, however, insufficient to guarantee
unique recovery of both Lam\'{e} parameters in the slab. In fact, we
require the Jost function as the \textit{spectral data}. We let $x$
denote the coordinate tangential to the boundary with Fourier dual
(wave vector) $\xi$. For fixed frequency, $\omega$, we use the
asymptotics as $\sqrt{\omega^2 \hat{\mu}_0^{-1} - \xi^2} \rightarrow
\infty$ of the Weyl matrix and the Jost function. We do require
\textit{two} distinct frequencies. Although the dependencies of
Lam\'{e} parameters and required data are different from the inverse
problem analyzed by Beals, Henkin and Novikova \cite{BealsCoifman,BealsCoifman2},
various steps in our proofs follow the logic of their proofs. To keep the presentation self-contained, we will review certain aspects
of the work of Beals, Henkin and Novikova.

The main results of this paper are the following:
\begin{itemize}
 \item We describe analytic properties of the Jost and Weyl solutions
   and functions for the Rayleigh Sturm-Liouville problem, in $\xi$,
   on the Riemann surface and study their asymptotic behaviors on the
   physical sheet.
\item
  Using the Wronskian for the solutions of the adjoint Sturm-Liouville problems, we
  derive a formula representing the Weyl matrix in terms of the Jost
  function; we relate the Jost function to the boundary matrix for the
  original Rayleigh system.
\item Following \cite{Bealsetal1995}, we derive a Gel'fand-Levitan
  type equation relating the Weyl matrix to the apparent potential of the
  Rayleigh Sturm-Liouville equation, and establish the uniqueness of
  its solution.
\item We show the unique recovery of the Lam\'{e} parameters in two
  steps: Determining the Markushevich substitution at the bottom of
  the slab and then recovering the potential as well as the Lam\'{e}
  parameters in the slab using two distinct frequencies.
\end{itemize}

\noindent
The Rayleigh system has been considered for many decades in
seismology, in particular with the aim to estimate Lam\'{e} parameters
from the observation of Rayleigh waves at a few frequencies
\cite{DormanEwing1962}. Empirically, seismologists have found that the
eigenvalues corresponding with these waves as the data are
insufficient and have considered additional types of data in the
absence of a mathematical understanding of this inverse problem. The
fixed-frequency Rayleigh system is an extraction from the
time-harmonic elastic-wave system of equations. The inverse boundary value problem
for time-harmonic elastic waves on a bounded, Lipschitz subdomain of
$\R^3$ has been studied before. Nakamura and Uhlmann \cite{NU,NU2} proved
uniqueness assuming that the Lam\'{e} parameters are $C^{\infty}$ and
that the shear modulus is close to a positive constant. Eskin and
Ralston \cite{ER} proved a related result. (In the context of the
analyses of inverse boundary value problems for time-harmonic waves,
we note that Complex Geometrical Optics solutions employed in these
are multidimensional generalizations of Jost solutions.) Beretta
\textit{et al.} \cite{BdHFVZ} proved uniqueness and Lipschitz
stability of this inverse problem when the Lam\'{e} parameters and the
density are assumed to be piecewise constant on a given domain
partition, with partial boundary data. Global uniqueness of the
inverse problem in dimension three assuming general Lam\'{e}
parameters remains an open problem. We note again that, here, the
Lam\'{e} parameters only depend on the boundary normal coordinate.
 
\section{Rayleigh system}

We consider the Rayleigh system associated with elastic surface
Rayleigh waves in isotropic media \cite{dHINZ}
\begin{equation} \label{Hamiltonian_Rayleigh}
   H_0(x,\xi) \left(\begin{array}{c} \varphi_1 \\[0.4cm] \varphi_3
   \end{array}\right)
   = \left(\begin{array}{c} \displaystyle
     -\frac{\partial}{\partial Z} \left(\hat{\mu}
     \frac{\partial\varphi_1}{\partial Z}\right)
     - \ii |\xi| \left(\frac{\partial}{\partial Z}(\hat{\mu} \varphi_3)
     + \hat{\lambda} \frac{\partial}{\partial Z}\varphi_3\right)
     + (\hat{\lambda} + 2 \hat{\mu}) |\xi|^2 \varphi_1
   \\[0.25cm] \displaystyle
   -\frac{\partial}{\partial Z} \left((\hat{\lambda} + 2 \hat{\mu})
   \frac{\partial\varphi_3}{\partial Z}\right)
   - \ii |\xi| \left(\frac{\partial}{\partial Z}(\hat{\lambda} \varphi_1)
   + \hat{\mu} \frac{\partial}{\partial Z} \varphi_1\right)
   + \hat{\mu} |\xi|^2 \varphi_3 \end{array}\right) .
\end{equation}
We will use the notation $$\widetilde{w} = \ma \varphi_1 \\[0.4cm]
\varphi_3 \am .$$ We denote the eigenvalues of $H_0(x,\xi)$ by $\L_j =
\L_j(x,\xi)$. These follow from solving ($Z<0$)
\begin{eqnarray}
  -\frac{\partial}{\partial Z}
  \hat{\mu} \frac{\partial\varphi_1}{\partial Z}
  - \ii |\xi| \left(\frac{\partial}{\partial Z}(\hat{\mu} \varphi_3)
  + \hat{\lambda} \frac{\partial}{\partial Z} \varphi_3\right)
  + (\hat{\lambda} + 2 \hat{\mu}) |\xi|^2 \varphi_1
  &=& \Lambda \varphi_1 ,
\label{Rayleigh1}\\
  -\frac{\partial}{\partial Z}
  (\hat{\lambda} + 2 \hat{\mu}) \frac{\partial\varphi_3}{\partial Z}
  - \ii |\xi| \left(\frac{\partial}{\partial Z}(\hat{\lambda}\varphi_1)
  + \hat{\mu} \frac{\partial}{\partial Z} \varphi_1\right)
  + \hat{\mu}|\xi|^2 \varphi_3
  &=& \Lambda \varphi_3 ,
\label{Rayleigh2}
\end{eqnarray}
supplemented with Neumann boundary conditions,
\begin{align}
  a_-(\widetilde{w}) :=&\ \ii \hat{\lambda} |\xi| \vphi_1(0^-)
  + (\hat{\lambda} + 2\hat{\mu}) \pdpd{\vphi_3}{Z}(0^-)
  = 0 ,
\label{Rayleighboundary1}\\
  b_-(\widetilde{w}) :=&\ \ii |\xi|\hat{\mu} \vphi_3(0^-)
  + \hat{\mu}\pdpd{\vphi_1}{Z}(0^-) = 0 .
\label{Rayleighboundary2}
\end{align}
We will set $\rho \equiv 1$ and simplify the notation,
\[
  \lambda = \hat{\lambda} ,\ \mu = \hat{\mu}\ (\rho = 1) .
\]
Rayleigh problem (\ref{Rayleigh1})-(\ref{Rayleighboundary2})
corresponds with \cite[(1.1), (1,1')]{Bealsetal1995} upon identifying
\begin{multline} \label{relationBF}
  x = -Z ,\quad w_1 = -\ii \varphi_1 ,\ w_2 = \varphi_3 ,
\\
  \chi_1 = b_-(w) = \ii b_-(\widetilde{w}) ,
  \ \chi_2 = a_-(w) = -a_-(\widetilde{w}) ,
  \quad \xi = |\xi | ,\ \omega^2 = \Lambda .
\end{multline}
We proceed using this notation \cite{Markushevich1994, Bealsetal1995}:
\begin{eqnarray}
  \dede{}{x} \Bigl(\mu \dede{w_1}{x} - \xi \mu w_2\Bigr)
  - \xi \lambda \dede{w_2}{x}
  + \bigl(\omega^2-\xi^2(\lambda+2\mu)\bigr) w_1 &=& 0 ,
\label{1rw(1.1)}\\
  \dede{}{x} \Bigl((\lambda + 2 \mu) \dede{w_2}{x}
  + \xi \lambda w_1\Bigr) + \xi \mu \dede{w_1}{x}
  + (\omega^2 - \xi^2 \mu) w_2 &=& 0 ,
\label{1rw(1.2)}
\end{eqnarray}
supplemented with the boundary conditions
\begin{eqnarray}
  \left.\left(\mu \dede{w_1}{x}
  - \xi \mu w_2\right)\right\vert_{x=0^+}
  &=& \chi_1(\xi) = b_-(w) ,
\label{1rw(1.3)}\\
  \left.\left((\lambda + 2 \mu) \dede{w_2}{x}
  + \xi \lambda w_1\right)\right\vert_{x=0^+} &=& \chi_2(\xi)
  = a_-(w) .
\label{1rw(1.4)}
\end{eqnarray}
We write $\chi = (\chi_1,\chi_2)^{\rm T}$ with $\chi =
\chi(\xi)$. From now on, to simplify notation, we will use $\xi$ for
both $|\xi| \in \R_+$ and its values in $\C$ following analytic
continuation.

\medskip\medskip

\noindent
We consider the case of an inhomogeneous isotropic elastic slab of
thickness $H$ bonded to a homogeneous isotropic elastic half space
with Lam\'e parameters $\lambda_0$ and $\mu_0$. We assume that the
layer's Lam\'e parameters, $\lambda$ and $\mu$, are three times
continuously differentiable and smoothly matched to the half-space
constants, $\lambda_0$ and $\mu_0$, respectively. In earlier papers
\cite{dHINZ}, we used the notation $\lambda_I$ and $\mu_I$ for
$\lambda_0$ and $\mu_0$, respectively, and $|Z_I| = H$.

\medskip\medskip

\begin{assumption} \label{ass:1}
We let $\mu \ge \alpha_0 > 0$, $2 \mu + 3 \lambda \ge \beta_0 > 0$,
$\lambda , \mu \in C^3(\R_+)$ and $\lambda(x) = \lambda_0$, $\mu(x) =
\mu_0$ for $x \ge H$.
\end{assumption}

\medskip\medskip

Assumption~\ref{ass:1} can be weakened to letting $\lambda \in
C^1(\R_+)$. It reflects the \textit{strong ellipticity} condition
\cite{Chen1991} as this appears in the existence and uniqueness of
solutions of the boundary value problem for time-harmonic elastic
waves. The parameters $\lambda$ and $\mu$ will be further restricted
through Assumption~\ref{ass:2} below.

\section{Markushevich transform to two adjoint matrix Sturm-Liouville
  problems}
\label{sec:Mark}

We perform an analogue of the calibration transform on the Rayleigh
system to obtain a matrix-valued (essentially non-diagonalizable)
Sturm-Liouville problem. We follow \cite{ArgatovIantchenko2019}. 

Based on the Pekeris substitution \cite{Pekeris1934}, it was shown by
Markushevich \cite{Markushevich1986, Markushevich1989, Markushevich1994} that the boundary value problem
  (\ref{1rw(1.1)})-(\ref{1rw(1.4)}) with $\chi_1 = \chi_2 = 0$ can be
  reduced to two matrix Sturm-Liouville problems with mutually
  transposed potentials and boundary conditions. Here, we briefly
  outline the transformations for arbitrary $\chi_1$ and $\chi_2$. For
  conciseness of notation while suppressing the coordinate dual to
  $\xi$, in the remainder of the analysis, we use a ${}'$ to denote
  differentiation with respect to $x$.

Let $G$ be a $2 \times 2$-matrix solving the Cauchy problem,
\begin{equation}
   G' = \frac{1}{2} L G ,\quad G(0) = I_2 ,
\label{1rw(2.1)}
\end{equation}
where $I_2$ is the unit matrix, and 
\begin{equation}
   L = \left(\begin{array}{cc}
              0 & -d \\
             -c &  0 \end{array}\right)\quad\text{with}\quad
   c = \frac{1}{g_0} \frac{\mu (\lambda + \mu)}{(\lambda + 2 \mu)} ,\
   d =- 2 g_0 \biggl(\frac{1}{\mu}\biggr)'' .
\label{1rw(2.2)}
\end{equation}
We have $\det G(x) = 1,$ cf.~\cite{Markushevich1986}. 
We adopt the notation of Markushevich \cite{Markushevich1989}, where
$g_0$ stands for an arbitrary positive constant. It is convenient to
put $g_0 = \mu_0$, which we do from now onwards.


By the substitution ($x \in [0,\infty)$)
\begin{equation} \label{MPtr}
   \gM^{-1}(F) = \left(\begin{array}{c}
         w_1 \\
         w_2 \end{array}\right)
\end{equation}
with
\begin{equation}
  \gM^{-1}
  = \left(\begin{array}{cc}
          \displaystyle \dede{}{x} & 1 \\
          -\xi & 0 \end{array}\right)
    \left(\begin{array}{cc}
    \displaystyle \frac{\mu_0}{\mu} & 0 \\
    0 & \displaystyle \frac{\mu}{\lambda+2\mu} \end{array}\right)
    \bigl(G^{\rm T}\bigr)^{-1}
\label{1rw(2.10)new}
\end{equation}
signifying the \textit{inverse} Markushevich transform, problem
(\ref{1rw(1.1)})-(\ref{1rw(1.4)}) reduces to the matrix
Sturm-Liouville form
\begin{eqnarray}
   F'' - \xi^2 F &=& Q F ,\quad\ x \in (0,\infty) ,
\label{1rw(2.11)new} \\
   F'\ + \Theta F &=& \bigl(D^{\rm a}\bigr)^{-1} \chi ,\quad
   x = 0 .
\label{1rw(2.12)new}
\end{eqnarray}
Here, $\Theta = \Theta(\xi) = (D^{\rm a}(\xi))^{-1} C^{\rm a}(\xi)$
with
\begin{multline}
   D^{\rm a}(\xi) = \left(\begin{array}{cc}
         \displaystyle -2 \mu_0 \frac{\mu'(0)}{\mu(0)} & \mu(0) \\
         -2 \mu_0 \xi & 0 \end{array}\right) ,
\\   
   C^{\rm a}(\xi) = \left(\begin{array}{cc}
     \displaystyle \mu_0 \biggl(2 \xi^2 - \frac{\omega^2}{\mu(0)}
                       + \frac{\mu''(0)}{\mu(0)}\biggr) &
     \displaystyle -\frac{\mu'(0) \mu(0)}{\lambda(0) + 2 \mu(0)} \\
     \displaystyle 2 \mu_0 \xi \frac{\mu'(0)}{\mu(0)} &
     \displaystyle -\xi \frac{\mu^2(0)}{\lambda(0) + 2 \mu(0)}
     \end{array}\right) ,
\label{1rw(2.14)new}
\end{multline}
so that 
\begin{equation} \label{Da}
  \left(D^{\rm a}(\xi)\right)^{-1}
  = \frac{1}{2 \mu_0 \mu \xi}
    \left(\begin{array}{cc} 0 & -\mu(0) \\
    2 \mu_0 \xi & \displaystyle -2 \mu_0 \frac{\mu'(0)}{\mu(0)} 
    \end{array}\right)
\end{equation}
and
\begin{align}
   \Theta(\xi) &= \left(\begin{array}{cc}
   \displaystyle -\frac{\mu'(0)}{\mu(0)} &
   \displaystyle \frac{1}{2 \mu_0}
        \frac{\mu^2(0)}{(\lambda(0)+2\mu(0))} \\
   \displaystyle \frac{\mu_0}{\mu(0)}
        \biggl(2 \xi^2 - \frac{\omega^2}{\mu(0)}
   - \mu(0) \biggl(\frac{1}{\mu}\biggr)''(0)\biggr) & 0
   \end{array}\right)
\label{1rw(2.15)new} \\
\nonumber \\
  &=: \ma -\theta_3 & \theta_2 \\ \\
      \displaystyle 2 \frac{\mu_0}{\mu(0)} \xi^2 - \theta_1 & 0 \am .
\nonumber
\end{align}
Furthermore, $Q$ is the matrix-valued potential given by
\begin{equation} \label{1rw(2.8)new}
   Q = \left(G^{-1} B G\right)^{\rm T} ,\quad B = B_1 + \omega^2 B_2 ,
\end{equation}
with
\begin{equation}
   B_1 = \left(\begin{array}{cc}
   \displaystyle -\frac{1}{2}\biggl(\frac{1}{\mu}\biggr)''
   \frac{\mu (\lambda + \mu)}{\lambda + 2 \mu}
   + \frac{\mu''}{\mu} &
   \displaystyle \mu_0\biggl(2\frac{\mu'}{\mu}
   \biggl(\frac{1}{\mu}\biggr)''
       + \biggl(\frac{1}{\mu}\biggr)'''\biggr)
   \\
   \displaystyle \frac{1}{\mu_0} \biggl(\frac{\lambda' \mu^2
   + \mu' \lambda (\lambda + \mu)}{(\lambda + 2 \mu)^2}
   - \frac{1}{2} \biggl(\frac{\mu (\lambda + \mu)}{
       \lambda + 2 \mu}\biggr)'\biggr) &
   \displaystyle \frac{1}{2} \biggl(\frac{1}{\mu}\biggr)''
     \frac{(\lambda - \mu)}{\lambda + 2 \mu}
   \end{array}\right) ,
\label{1rw(2.9a)new}
\end{equation}
\begin{equation}
   B_2 = \left(\begin{array}{cc}
     \displaystyle -\frac{1}{\mu} &
     \displaystyle \mu_0 \biggl(\frac{1}{\mu^2}\biggr)'
   \\
   \displaystyle 0 & \displaystyle -\frac{1}{\lambda + 2 \mu}
   \end{array}\right) .
\label{1rw(2.9b)new} 
\end{equation}
We note that the potential is not a symmetric matrix, that is, $Q \neq
Q^{\rm T}$.

By the adjoint substitution 
\begin{equation}
   \left(\gM^{\rm a}\right)^{-1}(F^{\rm a})
   = \left(\begin{array}{c}
         w_1 \\
         w_2 \end{array}\right)
\end{equation}
with
\begin{equation}
  \left(\gM^{\rm a}\right)^{-1}
  = \left(\begin{array}{cc}
    0 & -\xi \\ 1 & \displaystyle \dede{}{x}
    \end{array}\right)
    \left(\begin{array}{cc}
      1 & \displaystyle -2 \mu_0 \biggl(\frac{1}{\mu}\biggr)'
      \\
      0 & \displaystyle \frac{\mu_0}{\mu}\end{array}\right)
    G ,
\label{1rw(2.3)new}
\end{equation}
problem (\ref{1rw(1.1)})-(\ref{1rw(1.4)}) transforms to the matrix
Sturm-Liouville form
\begin{eqnarray}
  (F^{\rm a})'' - \xi^2 {F^{\rm a}} &=& Q^{\rm a} {F^{\rm a}} ,
  \quad x \in (0,\infty) ,
\label{1rw(2.4)new} \\
  (F^{\rm a})' + \Theta^{\rm a} {F^{\rm a}} &=&
           D^{-1} \chi ,\quad x = 0 .
\label{1rw(2.5)new}
\end{eqnarray}
Here,
\begin{equation}
    Q^{\rm a} =  Q^{\rm T} ,\quad
    \Theta^{\rm a} = \Theta^{\rm T} = \Theta^{\rm T}(\xi)
    = D^{-1}(\xi) C(\xi)
\end{equation}
is a $2 \times 2$-matrix with $D(\xi)$ and $C(\xi)$ being the
matrices,
\begin{multline}
   D(\xi) = \left(\begin{array}{cc}
       0 & -2 \xi \mu_0 \\[0.25cm]
       \mu(0) & 0 \end{array}\right) ,
\\
   C(\xi) = \left(\begin{array}{cc}
   \displaystyle -\xi \frac{\mu^2(0)}{(\lambda(0) + 2 \mu(0))} & 0
   \\
   -\mu'(0) & \displaystyle \frac{\mu_0}{\mu(0)}
   \biggl(2 \mu(0) \xi^2 - \omega^2 - 2 \frac{(\mu'(0))^2}{\mu(0)}
   + \mu''(0)\biggr) \end{array}\right) .
\label{1rw(2.7)new}
\end{multline}

\subsection*{Homogeneous half space, $x \in (H,\infty)$}

In components, (\ref{1rw(2.1)}) has the form 
\begin{equation}
   G_{11}' = -\frac{d}{2} G_{21} ,\quad
   G_{12}' = -\frac{d}{2} G_{22} ,\quad
   G_{21}' = -\frac{c}{2} G_{11} ,\quad
   G_{22}' = -\frac{c}{2} G_{12} ,
\label{1rw(3.1)new}
\end{equation}
in which, in view of (\ref{1rw(2.2)}), the coefficient $d$ is zero if
$\mu$ is constant. 
We consider the (homogeneous) half space $x \in (H,\infty)$ and write
\begin{equation}
   G_{11}(H) = G_{11}^H ,\quad
   G_{12}(H) = G_{12}^H ,\quad
   G_{21}(H) = G_{21}^H ,\quad
   G_{22}(H) = G_{22}^H .
\label{1rw(3.6)new}
\end{equation}
Then the matrix function, $G$, inside $x \in (H,\infty)$ can be
determined from the Cauchy problem
\begin{equation}
   G' = -\frac{c_0}{2} \left(\begin{array}{cc}
            0 & 0 \\[0.25cm] G_{11} & G_{12} \end{array}\right) ,\quad
   G(H) = \left(\begin{array}{cc}
            G_{11}^H & G_{12}^H \\[0.25cm]
            G_{21}^H & G_{22}^H \end{array}\right) ,
\label{1rw(3.7)new}
\end{equation}
in which
\begin{equation}
   c_0 = \frac{\lambda_0 + \mu_0}{\lambda_0 + 2 \mu_0} .
\label{1rw(3.8)new}
\end{equation}
The solution is
\begin{equation}
\begin{array}{c}
   \displaystyle \phantom{{}_{\Bigr)}}
   G_{11}(x) = G_{11}^H ,\quad G_{12}(x) = G_{12}^H ,
   \phantom{{}_{\Bigr)}} \\
   \displaystyle 
   G_{21}(x) = -\frac{c_0}{2} G_{11}^H(x - H) + G_{21}^H ,\quad
   G_{22}(x) = -\frac{c_0}{2} G_{12}^H(x - H) + G_{22}^H .
\end{array}
\label{1rw(3.9)new}
\end{equation}
As $\det G(x) = 1$ (see \cite{Markushevich1986, Markushevich1989,
  Markushevich1994}), the inverse matrix follows to be
\begin{equation}
   G^{-1} = \left(\begin{array}{rr}
                G_{22} & -G_{12} \\[0.25cm]
               -G_{21} & G_{11} \end{array}\right) .
\label{1rw(3.10)new}
\end{equation}
Thus, in the homogeneous elastic half space, $x \in (H,\infty)$,
according to (\ref{1rw(2.8)new})-(\ref{1rw(2.9b)new}) and
(\ref{1rw(3.10)new}), we have
\begin{equation} \label{1rw(3.11)new}
   Q = \omega^2 \left(\begin{array}{cc}
   \displaystyle \frac{G_{12} G_{21}}{\lambda_0 + 2 \mu_0}
   - \frac{G_{11}G_{22}}{\mu_0} &
   \displaystyle G_{11} G_{21} \frac{c_0}{\mu_0}
   \\[0.25cm]
   \displaystyle -G_{12} G_{22} \frac{c_0}{\mu_0} &
   \displaystyle \frac{G_{12} G_{21}}{\mu_0}
                 - \frac{G_{11} G_{22}}{\lambda_0 + 2 \mu_0}
   \end{array}\right) ,
\end{equation}
where the components of the transformation matrix $G$ are given by
(\ref{1rw(3.9)new}). It is of interest to observe that if $G_{12}^H
\not= 0$, then all components of the potential matrix, $Q$, will have
linear growth as $x \to \infty$.

From here onwards, we denote $Q(x)$ for $x \geq H$ by $Q_0(x)$.
Using, again, that $\det G(x) = 1$, we obtain
\begin{align}
  Q_0(x) =&
  \omega^2 \ma \displaystyle -\frac{1}{\mu_0} & 0 \\
  0 & \displaystyle -\frac{1}{\lambda_0 + 2 \mu_0} \am
  + \omega^2 \frac{\lambda_0 + \mu_0}{\mu_0 (\lambda_0 + 2 \mu_0)}
  \ma -G_{12} G_{21} & G_{21}G_{11} \\[0.25cm]
      -G_{12} G_{22} & G_{12}G_{21} \am
\label{Q0} \\[0.25cm]
  =& \omega^2 \ma \displaystyle - \frac{1}{\mu_0} & 0 \\
  0 & \displaystyle -\frac{1}{\lambda_0 + 2 \mu_0} \am
\nonumber\\[0.25cm] &\qquad
  + \omega^2\frac{c_0}{\mu_0}
  \ma \displaystyle
  -G_{12}^H \left[-\frac{c_0}{2}G_{11}^H(x - H) + G_{21}^H\right] &
  \displaystyle 
  G_{11}^H \left[-\frac{c_0}{2} G_{11}^H(x - H) + G_{21}^H\right]
  \\[0.25cm]
  \displaystyle
  -G_{12}^H \left[-\frac{c_0}{2} G_{12}^H(x - H) + G_{22}^H\right] &
  \displaystyle 
  G_{12}^H \left[-\frac{c_0}{2} G_{11}^H(x - H) + G_{21}^H\right]
  \am .
\nonumber
\end{align}
We extend $Q_0 = Q_0(x)$ to $x \in (0,H]$ linear in $x$, and refer to
it as the background potential. Then we introduce the perturbation
potential $$V(x) = Q(x) - Q_0(x)$$ so that $V(x) = 0$ for $x \geq H$.

\medskip\medskip
  
\begin{remark}
If $H = 0$ then $G_{12}^H = G_{21}^H = 0$, $G_{11}^H = G_{22}^H = 1$,
and
\begin{equation*}
  Q_0(x) = \omega^2 \ma \displaystyle
  -\frac{1}{\mu_0} & 0 \\ 0 & \displaystyle
  -\frac{1}{\lambda_0 + 2 \mu_0} \am
  + \omega^2 \frac{c_0^2}{2 \mu_0}
                        \ma 0 & 1 \\[0.25cm] 0 & 0 \am \, x .
\end{equation*}
\end{remark}

\noindent
We write
\[
   \varpi = \frac{\mu_0}{\mu(0)}
\]
and introduce a class of potentials

\medskip\medskip

\begin{definition} \label{def:Q}
A real matrix-valued potential, $Q$, is of Lam\'{e} type if it can be
generated from Lam\'{e} parameters according to the Markushevich
transform, that is, is of the form
(\ref{1rw(2.8)new})-(\ref{1rw(2.9b)new}). Due to
Assumption~\ref{ass:1}, $Q \in C^1(\R_+) \cap L^{\infty}(\R_+)$.
\end{definition}

\medskip\medskip

\noindent
In view of Assumption~\ref{ass:1}, our matrix-valued potential, $Q$,
of Lam\'{e} type attains the form $Q_0$ on $[H,\infty)$ as in
  (\ref{Q0}). Then $V = Q - Q_0 \in L^1([0,H])$.

The Lam\'{e} parameters at $x = 0$ and $x \ge H$, that is,
\emph{$\lambda(0)$, $\mu(0)$ as well as $\mu'(0)$, $\mu''(0)$ and
  $\lambda_0$ and $\mu_0$ are encoded in, and determine $\Theta$ (or
  $\theta_1$, $\theta_2$ and $\theta_3$)} independently of $Q$. In
this paper, we will not consider the problem of boundary
determination.

\section{Jost and Weyl solutions, Jost function and Weyl matrix}

We introduced $V$ in (\ref{1rw(2.11)new}) and obtain
\begin{equation} \label{StLgen_bis}
  -F'' + Q_0 F + V F =-\xi^2 F ,\quad x \in (0,\infty)
\end{equation}
supplemented with (\ref{1rw(2.12)new}),
\begin{equation} \label{1rw(2.12)bis}
   F' + \Theta F = \bigl(D^{\rm a}\bigr)^{-1} \chi ,\quad
   x = 0 ,
\end{equation}
where $\Theta$ is given in (\ref{1rw(2.15)new}).

First, we construct solutions to ``background'' equation,
\begin{equation} \label{StLgen_0}
  -F'' + Q_0 F = -\xi^2 F
\end{equation}
of the form
\begin{align*}
  F^\pm_{P,0} =& \ma F^\pm_{P,0,1} \\  F^\pm_{P,0,2} \am
                e^{\pm \ii x q_P} ,\quad
  q_P = \sqrt{\frac{\omega^2}{\lambda_0 + 2 \mu_0} - \xi^2} ,
\\
  F^\pm_{S,0} =& \ma F^\pm_{S,0,1} \\  F^\pm_{S,0,2} \am
                e^{\pm \ii x q_S} ,\quad
  q_S = \sqrt{\frac{\omega^2}{\mu_0} - \xi^2}
\end{align*}
so that their inverse Markushevich transforms
($\gM^{-1}(F^\pm_{P,0})$, $\gM^{-1}(F^\pm_{S,0})$;
cf.~(\ref{1rw(2.10)new})) are proportional to
\begin{equation} \label{Jost0or}
  \left(\begin{array}{c} \displaystyle 1 \\
        \displaystyle \mp \frac{\ii}{\xi}q_P
  \end{array}\right) e^{\pm \ii x q_P} ,\quad
  \left(\begin{array}{c} \displaystyle \mp \frac{\ii}{\xi} q_S \\
    \displaystyle 1 \end{array}\right) e^{\pm \ii x q_S} ,
\end{equation}
respectively. We make this precise below. We refer to $q_P$ and $q_S$
as quasi-momenta. We note that (\ref{Jost0or}) are solutions to the
Rayleigh system (\ref{1rw(1.1)})-(\ref{1rw(1.2)}) with $\lambda =
\lambda_0, \mu = \mu_0$ constant for all $x \geq 0$. We may construct
similar solutions to the adjoint equation,
\begin{equation}
  -(F^{\rm a})'' + Q_0^{\rm T} F^{\rm a} = -\xi^2 F^{\rm a} .
\end{equation}

We consider (\ref{1rw(2.11)new}) and (\ref{1rw(2.4)new}) on $x \in
(H,\infty)$, where the potential $Q$ is given by formula (\ref{Q0})
with the transformation matrix $G$ determined by
(\ref{1rw(3.9)new}). Using that
\[
\left(\begin{array}{cc}
\displaystyle \dede{}{x} & 1 \\[0.25cm]
-\xi & 0
\end{array}\right)^{-1}
= \left(\begin{array}{cc}
0 & \displaystyle -\frac{1}{\xi} \\[0.25cm]
1 & \displaystyle \frac{1}{\xi} \dede{}{x}
\end{array}\right) ,\quad
\left(\begin{array}{cc}
0 & -\xi \\[0.25cm]
1 & \displaystyle \dede{}{x}
\end{array}\right)^{-1}
= \left(\begin{array}{cc}
\displaystyle \frac{1}{\xi} \dede{}{x} & 1 \\[0.25cm]
\displaystyle -\frac{1}{\xi} & 0
\end{array}\right) ,
\]
(\ref{1rw(2.10)new}), for $x \in (H,\infty)$, implies that
\begin{equation}
  F = \gM(w)
  = G^{\rm T} \left(\begin{array}{cc}
     \displaystyle 1 & 0 \\
     0 & \displaystyle \frac{\lambda_0 + 2 \mu_0}{\mu_0}
     \end{array}\right) \left(\begin{array}{cc}
     0 & \displaystyle -\frac{1}{\xi} \\
     1 & \displaystyle \frac{1}{\xi} \dede{}{x}
     \end{array}\right) w ,
\label{1rw(8.3)new}
\end{equation}
while (\ref{1rw(2.3)new}), for $x \in (H,\infty)$, implies that
\begin{equation}
  {F^{\rm a}} = \gM^{\rm a}(w)
  = G^{-1} \left(\begin{array}{cc}
     \displaystyle \frac{1}{\xi} \dede{}{x} & 1 \\
     \displaystyle -\frac{1}{\xi} & 0
     \end{array}\right) w ,
\label{1rw(8.2)new}
\end{equation}
where $w = (w_1,w_2)^{\rm T}$ solves the Rayleigh system
(\ref{1rw(1.1)})--(\ref{1rw(1.2)}). For $x \in (H,\infty)$ we have
linearly independent solutions
\begin{equation}
   w^\pm_{P,0} = \mu_0 \frac{\xi^2}{\omega^2}
   \left(\begin{array}{c} \displaystyle 1 \\
         \displaystyle \mp \frac{\ii}{\xi} q_P
   \end{array}\right) e^{\pm \ii x q_P} ,
\label{1rw(8.5)new}
\end{equation}
\begin{equation}
   w^\pm_{S,0} = \mu_0 \frac{\xi^2}{\omega^2}
   \left(\begin{array}{c} \displaystyle \mp \frac{\ii}{\xi} q_S \\
   \displaystyle 1
   \end{array}\right) e^{\pm \ii x q_S} ,
\label{1rw(8.4)new}
\end{equation}
which correspond to the solutions of the original Rayleigh system
(\ref{Rayleigh1})--(\ref{Rayleigh2})
\begin{equation}
   \widetilde{w}^\mp_{P,0} = \ii \mu_0 \frac{\xi}{\omega^2}
   \left(\begin{array}{c} \displaystyle \xi
     \\ \displaystyle \mp  q_P \end{array} \right) e^{\mp iZq_P}
   =: \ii \mu_0 \frac{\xi}{\omega^2} f^\mp_{P,0} ,
\label{1rw(8.5)newold}
\end{equation}
\begin{equation}
   \widetilde{w}_{S,0}^{\mp} = -\mu_0 \frac{\xi}{\omega^2}
   \left(\begin{array}{c} \displaystyle \mp q_S
     \\ \displaystyle -\xi \end{array}\right) e^{\mp iZq_S}
   =: -\mu_0 \frac{\xi}{\omega^2} f_{S,0}^{\mp} .
\label{1rw(8.4)newold}
\end{equation}
Using (\ref{1rw(3.9)new}), substitution of (\ref{1rw(8.5)new}) and
(\ref{1rw(8.4)new}) into (\ref{1rw(8.3)new}) and (\ref{1rw(8.2)new}),
yields
\begin{eqnarray} \label{1rw(8.7)new}
  F^\pm_{P,0} = \gM(w^\pm_{P,0}) &=&
  \left(\begin{array}{c}
    \displaystyle -\frac{c_0}{2} G_{11}^H(x - H) + G_{21}^H
    \pm \ii q_P \frac{\mu_0}{\omega^2} G_{11}^H  \\[0.25cm]
    \displaystyle -\frac{c_0}{2} G_{12}^H(x - H) + G_{22}^H
    \pm \ii q_P \frac{\mu_0}{\omega^2} G_{12}^H
  \end{array}\right) e^{\pm \ii x q_P} ,
\\ \label{1rw(8.7bis)new}
  F^\pm_{S,0} = \gM(w^\pm_{S,0}) &=&
  -\mu_0 \frac{\xi}{\omega^2}
    \left(\begin{array}{c}
    \displaystyle G_{11}^H \\[0.25cm]
    \displaystyle G_{12}^H \end{array}\right) e^{\pm \ii x q_S}
\end{eqnarray}
and
\begin{eqnarray} \label{1rw(8.6)new}
  F^{{\rm a},\pm}_{S,0} = \gM^{\rm a}(w_{S,0}^{\pm}) &=&
  \left(\begin{array}{c}
    \displaystyle -\frac{c_0}{2} G_{12}^H(x - H) + G_{22}^H
    \mp \ii q_S \frac{\mu_0}{\omega^2} G_{12}^H \\[0.25cm]
    \displaystyle \frac{c_0}{2} G_{11}^H(x - H) - G_{21}^H
    \pm \ii q_S \frac{\mu_0}{\omega^2} G_{11}^H
  \end{array}\right) e^{\pm \ii x q_S} ,
\\ \label{1rw(8.6bis)new}
  F^{{\rm a},\pm}_{P,0} = \gM^{\rm a}(w_{P,0}^{\pm}) &=&
  \mu_0 \frac{\xi}{\omega^2} \left(\begin{array}{c}
    \displaystyle G_{12}^H \\[0.25cm]
    \displaystyle -G_{11}^H 
  \end{array}\right) e^{\pm \ii x q_P} .
\end{eqnarray}

\medskip\medskip

\begin{remark}
If $H = 0$, substituting $G_{12}^H = G_{21}^H = 0$, $G_{11}^H =
G_{22}^H = 1$ into (\ref{1rw(8.7)new})-(\ref{1rw(8.7bis)new}) and
(\ref{1rw(8.6)new})-(\ref{1rw(8.6bis)new}), respectively, yields
\[
  F^\pm_{P,0} = \left(\begin{array}{c}
    \displaystyle -\frac{c_0}{2} x
    \pm \ii q_P \frac{\mu_0}{\omega^2} \\ 
    1 \end{array}\right) e^{\pm \ii x q_P} ,\quad
  F^\pm_{S,0} = \ma \displaystyle
    -\mu_0 \frac{\xi}{\omega^2} \\ 0 \am e^{\pm \ii x q_S}
\]
and
\[
  F^{{\rm a}, \pm}_{S,0} = \ma 1 \\
    \displaystyle \frac{c_0}{2} x
    \pm \ii q_S \frac{\mu_0}{\omega^2} \am e^{\pm \ii x q_S} ,\quad
  F^{{\rm a}, \pm}_{P,0} = \left(\begin{array}{c} 0 \\
    \displaystyle -\mu_0 \frac{\xi}{\omega^2}
  \end{array}\right) e^{\pm \ii x q_P} .
\]
\end{remark}

\medskip\medskip

\noindent
In view of the presence of square roots, $q_p$ and $q_S$, we introduce
the cut complex plane,
\[
   \cK = \C \setminus
   \left(\left[-\frac{\omega}{\sqrt{\mu_0}},
   \frac{\omega}{\sqrt{\mu_0}}\right]\cup \ii \R\right) .
\]
In Appendix~\ref{app:RS} we introduce the corresponding Riemann
surface and physical Riemann sheet, $\cK_+$, by the condition $\Im
q_P(\xi) > 0$, $\Im q_S(\xi) > 0$.

The Jost solutions, $F^\pm_{P}, F^\pm_{S}$, of (\ref{StLgen_bis}) are
determined by the conditions
\[
   F^\pm_P = F^\pm_{P,0} ,\quad
   F^\pm_S = F^\pm_{S,0}\quad\text{for}\quad x \geq H .
\]
We define the matrix Jost solutions, $\bF = \bF(x,\xi)$ and $\bF_0 =
\bF_0(x,\xi)$ (emphasizing that, here, $\xi$ is \textit{not} the
Fourier dual to $x$), as
%
%
\begin{equation} \label{eq:Jostsoldef}
   \bF(x,\xi) = [F^+_P \,\, F^+_S] ,\quad
   \bF_0(x,\xi) = [F^+_{P,0} \,\, F^+_{S,0} ] ,
\end{equation}
and the Jost function as (cf.~(\ref{1rw(2.12)new}))
\[
   \bF_\Theta(\xi) = \bF'(0,\xi) + \Theta(\xi) \bF(0,\xi) ,
\]
where $\Theta(\xi)$ is given in (\ref{1rw(2.15)new}).

\medskip\medskip

\begin{remark}
By (\ref{eq:FThetBf}), we have $\det\bF_\Theta(\xi) = c \xi \,
\Delta_{\mathrm{R}}$, where $\Delta_{\mathrm{R}}$ is the Rayleigh
determinant (see Remark~\ref{rem:Whhsp}) and $c$ is a constant. The
Rayleigh determinant originates from the reflection matrix that we
will introduce later.
\end{remark}

\medskip\medskip

\noindent
The Jost solutions, $F^{{\rm a},\pm}_{P}, F^{{\rm a},\pm}_{S}$, of
\begin{equation} \label{StLgen_bis_con}
  -(F^{\rm a})'' + Q_0^{\rm T} F^{\rm a} + V^{\rm T} F^{\rm a}
    = -\xi^2 F^{\rm a} ,\quad x \in (0,\infty)
\end{equation}
(upon introducing $V$ in (\ref{1rw(2.4)new})) are determined by the
conditions
\[
   F^{{\rm a},\pm}_P = F^{{\rm a},\pm}_{P,0} ,\quad
   F^{{\rm a},\pm}_S = F^{{\rm a},\pm}_{S,0}\quad\text{for}\quad x > H .
\]
We define the matrix Jost solutions, $\bF^{\rm a} = \bF^{\rm
  a}(x,\xi)$ and $\bF^{\rm a}_0 = \bF^{\rm a}_0(x,\xi)$, as
\begin{equation} \label{eq:Jostsoladef}
   \bF^{\rm a}(x,\xi) = [F^{{\rm a},+}_P \,\, F^{{\rm a},+}_S] ,\quad
   \bF^{\rm a}_0(x,\xi) = [F^{{\rm a},+}_{P,0} \,\,
                          F^{{\rm a},+}_{S,0}] ,
\end{equation}
and the Jost function as (cf.~(\ref{1rw(2.5)new}))
\[
   \bF^{\rm a}_\Theta(\xi) = (\bF^{\rm a})'(0,\xi)
           + \Theta^{\rm a}(\xi) \bF^{\rm a}(0,\xi) .
\]
Using (\ref{eq:FThetBf})-(\ref{eq:FaThetBf}), we find that
\begin{equation} \label{Jfunction-adjoint}
   \bF_\Theta^{\rm a}(\xi) = \ma \displaystyle{
     -2 \frac{\mu_0}{\mu_0} \xi} & 0 \\ \displaystyle{
      \frac{\mu'(0)}{\mu(0)} \frac{1}{\xi}} & \displaystyle{
      -\frac{\mu(0)}{2 \mu_0} \frac{1}{\xi}} \am \bF_\Theta(\xi) .
\end{equation}

We note that $\bF^{\rm a} = (\bF^{\star})^{\rm T}$, where
$\bF^{\star}$ denotes the solution to the adjoint problem according to
\cite{Bondarenko2015}.The Wronskian of the Jost
  solutions of the adjoint problems has the familiar property of being
  independent of $x$
\begin{equation} \label{Wronskian_property_third}
   \frac{\mathrm{d}}{\mathrm{d}x} W({\bF^{\rm a}},\bF) = 0 ;
\end{equation}
we obtain

\medskip\medskip

\begin{lemma} \label{lem:mwronskian}
Let $\bF$, $\bF_0$, $\bF^{\rm a}$ and $\bF^{\rm a}_0$ be given by
(\ref{eq:Jostsoldef}) and (\ref{eq:Jostsoladef}), respectively. Then
\begin{multline} \label{wronskian-xi}
   W(\bF^{\rm a}(x,-\xi),\bF(x,\xi))
             = W(\bF^{\rm a}_0(x,-\xi),\bF_0(x,\xi))
\\
   = ((\bF_0^{\rm a})')^{\rm T}(x,-\xi) \bF_0(x,\xi)
            - (\bF_0^{\rm a})^{\rm T}(x,-\xi) \bF_0'(x,\xi)
   = -\ii \, 2 \mu_0 \frac{\xi}{\omega^2}
          \ma q_P & 0 \\[0.25cm] 0 & -q_S \am .
\end{multline}
\end{lemma}


\medskip\medskip

\noindent
Now, we define the Weyl solution, $\boldsymbol{\Phi}$, as
\begin{equation} \label{eq:defWsol}
  \boldsymbol{\Phi}(x,\xi) = \bF(x,\xi) [\bF_\Theta (\xi)]^{-1}
\end{equation}
and the Weyl matrix, $\bM$, as
\begin{equation} \label{eq:Mdef}
   \bM(\xi) = \boldsymbol{\Phi}(0,\xi)
            = \bF(0,\xi) [\bF_\Theta (\xi)]^{-1} .
\end{equation}
This definition shows that $\bM(\xi) \bF_\Theta (\xi) = \bF(0,\xi)$,
whence $\bM(\xi)$ can be identified with the
\textbf{Robin-to-Dirichlet map} associated with the matrix
Sturm-Liouville problem (\ref{1rw(2.11)new}). Clearly, also
\begin{equation} \label{eq:PhiM}
   \boldsymbol{\Phi}(x,\xi) = \bF(x,\xi) [\bF(0,\xi)]^{-1} \bM(\xi) .
\end{equation}

\medskip\medskip

\begin{remark} \label{rem:Whhsp}
In a homogeneous half space, explicit calculations result in
\begin{equation} \label{eq:Mhhsp}
  \bM(\xi) = \frac{\mu_0 \xi}{\omega^2 \Delta_0(\xi)}
  \frac{1}{\ii} \ma \ii q_P & \displaystyle \frac{1}{2}
  - \frac{\mu_0}{\omega^2} \xi^2 - \frac{\mu_0}{\omega^2} q_P q_S \\
  \displaystyle \frac{\omega^2}{\mu_0} - 2 \xi^2 & \ii q_S \am ,
\end{equation}
where $\Delta_0 = \det \bF_{0,\Theta} = -\frac{\mu_0^2}{2 \omega^4} \,
\xi \Delta_{\mathrm{R}}$ with
(cf. (\ref{1rw(8.5)newold})--(\ref{1rw(8.4)newold}))
\begin{equation} \label{eq:Rdet0}
  \Delta_{\mathrm{R}}(\xi) = \left(\mstrut{0.5cm}\right.
  \left(\frac{\omega^2}{\mu_0} - 2 \xi^2\right)^2
        + 4 \xi^2 q_P q_S\left.\mstrut{0.5cm}\right)
  = -\frac{1}{\mu_0^2} \det
    \ma a_-(f^-_{P,0}) & a_-(f^-_{S,0}) \\
        b_-(f^-_{P,0}) & b_-(f^-_{S,0})\am .
\end{equation}
\end{remark}

\medskip\medskip

\noindent
We have
\begin{equation} \label{eq:PhiThe0-1}
  \boldsymbol{\Phi}'(0,\xi) + \Theta(\xi) \boldsymbol{\Phi}(0,\xi)
     = I_2 .
\end{equation}
In a similar fashion, we introduce
\begin{equation} \label{eq:defWasol}
   \boldsymbol{\Phi}^{\rm a}(x,\xi)
         = \bF^{\rm a}(x,\xi) [\bF^{\rm a}_\Theta (\xi)]^{-1}
\end{equation}
and the Weyl matrix, $\bM^{\rm a} = \bM^{\rm a}(\xi)$, as
\[
  \bM^{\rm a}(\xi) = \boldsymbol{\Phi}^{\rm a}(0,\xi)
              = \bF^{\rm a}(0,\xi) [\bF^{\rm a}_\Theta (\xi)]^{-1} .
\]
We have
\begin{equation} \label{eq:PhiaThe0-1}
  (\boldsymbol{\Phi}^{\rm a})'(0,\xi)
     + \Theta^{\rm a}(\xi) \boldsymbol{\Phi}(0,\xi) = I_2 .
\end{equation}
We make the following observation. Using (\ref{eq:WvphiS}) and
(\ref{eq:WavphiS}), we evaluate the Wronskian,
\begin{equation}
  W(\boldsymbol{\Phi}^{\rm a},\boldsymbol{\Phi})
  = W(\boldsymbol{\Phi}^{\rm a},\boldsymbol{\Phi})\Big|_{x=0}
  = \bM - (\bM^{\rm a})^{\rm T} .
\end{equation}
As, using the expressions for $\bF_0$ and $\bF^{\rm a}_0$ and
independence of the Wronskian of $x$,
\[
   \lim_{x \rightarrow \infty} W(\boldsymbol{\Phi}^{\rm a},
   \boldsymbol{\Phi}) = 0\quad\text{for $\xi \in \cK_+$} ,
\]
we conclude that
\begin{equation} \label{eq:MaM}
   \bM^{\rm a} = \bM^{\rm T} .
\end{equation}

\subsection*{Other solutions}

Following \cite{Bondarenko2015}, we introduce two other matrix-valued
solutions, $\boldsymbol{\varphi}(x,\xi)$, $\bS(x,\xi)$, of
(\ref{1rw(2.11)new}), that is, solutions to
\begin{equation} \label{StLunpert_matrix}
  -\bF'' + Q \bF = -\xi^2 \bF ,
\end{equation}
satisfying
\[
   \boldsymbol{\varphi}(0,\xi) = I_2,\ 
   \boldsymbol{\varphi}'(0,\xi) = -\Theta(\xi) ,\quad
   \bS(0,\zeta) = \mathbf{0} ,\
   \bS'(0,\xi) = I_2 . 
\]
Hence, $\boldsymbol{\varphi}$ satisfies the Robin-type boundary
condition
\begin{equation} \label{1rw(2.12)bis_hom}
  \mathbf{F}' + \Theta \mathbf{F} = 0 ,\quad x = 0 .
\end{equation}
Then the Weyl solution takes the form
\begin{equation} \label{eq:WvphiS}
   \boldsymbol{\Phi}(x,\xi) = \bS(x,\xi)
                + \boldsymbol{\varphi}(x,\xi) \bM(\xi) .
\end{equation}
Furthermore, we introduce the Green's function, $\bcG = \bcG(x,y)$, $0
< x < y$, satisfying
\begin{equation} \label{StLunpert_bis}
  -\bcG'' + Q_0 \bcG = -\xi^2 \bcG
\end{equation}
supplemented with
\[
   \bcG(y,y) = \mathbf{0} ,\quad
   \frac{\partial}{\partial x} \bcG(x,y) \bigg|_{x=y} = I_2 ;
\]
see Appendix~\ref{app:GF} for explicit expressions for $\bcG$. From
the definition it follows that the Green's function is entire in $\xi
\in \C$ (for any $\omega \in \C$ fixed). The matrix Jost solution
$\bF(x,\xi)$ then satisfies the Volterra-type integral equation
\begin{equation} \label{Volterra1}
  \bF(x,\xi) = \bF_0(x,\xi) - \int_x^H \bcG(x,y) V(y) \bF(y,\xi)
  \, \dd y .
\end{equation}

\medskip\medskip

In a similar fashion, we introduce two other matrix-valued solutions,
$\boldsymbol{\varphi}^{\rm a}(x,\xi)$, $\bS^{\rm a}(x,\xi)$, of
(\ref{1rw(2.4)new}), that is, solutions to
\begin{equation}
  -(\bF^{\rm a})'' + Q^{\rm a} \bF^{\rm a} = -\xi^2 \bF^{\rm a}
\end{equation}
satisfying
\[
   \boldsymbol{\varphi}^{\rm a}(0,\xi) = I_2 ,\ 
   (\boldsymbol{\varphi}^{\rm a})'(0,\xi) = -\Theta^{\rm a}(\xi) ,
   \quad \bS^{\rm a}(0,\xi) = \mathbf{0} ,\
   (\bS^{\rm a})'(0,\xi) = I_2 . 
\]
Hence, $\boldsymbol{\varphi}^{\rm a}$ satisfies the Robin-type
boundary condition
\begin{equation}
  (\mathbf{F}^{\rm a})' + \Theta^{\rm a} \mathbf{F}^{\rm a} = 0 ,
  \quad x = 0 .
\end{equation}
Then the Weyl solution takes the form
\begin{equation} \label{eq:WavphiS}
   \boldsymbol{\Phi}^{\rm a}(x,\xi) = \bS^{\rm a}(x,\xi)
        + \boldsymbol{\varphi}^{\rm a}(x,\xi) \bM^{\rm a}(\xi) .
\end{equation}
Furthermore, we introduce the Green's function, $\bcG^{\rm a} =
\bcG^{\rm a}(x,y)$, $0 < x < y$, satisfying
\begin{equation}
  -(\bcG^{\rm a})'' + Q_0^{\rm T} \bcG^{\rm a} = -\xi^2 \bcG^{\rm a}
\end{equation}
supplemented with
\[
   \bcG^{\rm a}(y,y) = \mathbf{0} ,\quad
   \frac{\partial}{\partial x} \bcG^{\rm a}(x,y) \bigg|_{x=y}
                    = I_2 .
\]
The matrix Jost solution $\bF^{\rm a}(x,\xi)$ then satisfies the
Volterra-type integral equation
\begin{equation} \label{Volterra2}
  \bF^{\rm a}(x,\xi) = \bF^{\rm a}_0(x,\xi)
      - \int_x^H \bcG^{\rm a}(x,y) [V(y)]^{\rm T} \bF^{\rm a}(y,\xi)
  \, \dd y .
\end{equation}

\section{Spectral data}
\label{ssec:Wspec}

\subsection*{Analytic continuation}

We note that
\[
   q_S(-\xi) = -q_S(\xi) ,\quad
   \xi \in \cK ,
\]
with an extension to the branch cuts,
\[
   \left(\left[-\frac{\omega}{\sqrt{\mu_0}},
     \frac{\omega}{\sqrt{\mu_0}}\right] \cup \ii \R\right) ;
\]
see Appendix~\ref{app:RS}. We give conjugation properties of the
matrix Jost solutions, in

\medskip\medskip

\begin{lemma} \label{l-conj}
For ${\displaystyle \xi \in \cK = \C \setminus
  \left(\left[-\frac{\omega}{\sqrt{\mu_0}},
    \frac{\omega}{\sqrt{\mu_0}}\right]\cup \ii \R\right)}$ (see
Appendix~\ref{app:RS}) the following holds true
\begin{align}
  & \overline{\bF(x,\xi)} = \bF(x,\overline{\xi}) ,\quad
    \overline{\bF^{\rm a}(x,\xi)} = \bF^{\rm a}(x,\overline{\xi}) ,
\label{conjugations0} \\
  & \bF(x,-\xi) = [F^+_P(x,-\xi) \,\, F^+_S(x,-\xi)]
          = [F^-_P(x,\xi) \,\, -F^-_S(x,\xi)] ,
\label{conjugations1} \\
  & \bF^{\rm a}(x,-\xi)
    = [F^{{\rm a},+}_P(x,-\xi) \,\, F^{{\rm a},+}_S(x,-\xi)]
    = [-F^{{\rm a},-}_P(x,\xi) \,\, F^{{\rm a},-}_S(x,\xi)] .
\label{conjugations2}
\end{align}
\end{lemma}

\begin{proof}
These properties are satisfied by the reference Jost solutions $\bF_0$
and $\bF^{\rm a}_0.$ Then we use the Volterra type integral equations
(\ref{Volterra1}), (\ref{Volterra2}) and the properties of the kernels
(as even functions in both $q_S$ and $q_P$) and identities
(\ref{k-properties10}).
\end{proof}

\medskip\medskip

\noindent
The conjugation properties in Lemma~\ref{l-conj} also imply that
\begin{equation} \label{Phiextension}
   \boldsymbol{\Phi}(\xi) = \overline{   \boldsymbol{\Phi}(\overline{\xi})}
\end{equation}
and
\begin{equation} \label{Mextension}
   \bM(\xi) = \overline{\bM(\overline{\xi})} .
\end{equation}
on $\cK_+$.

We observe that $\bM$ has a meromorphic continuation from the physical
(``upper'') sheet $\cK_+$ through the cuts to the unphysical
(``lower'') sheets and whole Riemann surface $\cR$ (see
Appendix~\ref{app:RS}), which still satisfies this conjugation
property.

\medskip\medskip

\noindent
We will simplify the analysis in the next subsection by introducing
\[
   \cK_+ \to \Pi_+ ,\quad \xi \to \zeta = \xi^2 ,
\]
below, and note that this also defines the inverse, $\Pi_+ \ni \zeta
\to \xi = \sqrt{\zeta} \in \cK_+$.

\subsection{Cauchy integral}
\label{ssec:CauchyI}

From asymptotic expansion of the Weyl matrix which we develop in
Lemma~\ref{lem:Wasymp} below, specifically (\ref{(2.4)}), it follows
that
\[
   \det \bM(\xi) = \frac{1}{\xi^2}
      \frac{\lambda(0) + 2 \mu(0)}{\lambda(0) + \mu(0)}
      + {\mathcal O}\left(\frac{1}{\xi^3}\right)\quad\text{as}\quad
      |\xi| \to \infty ,\ \xi \in \cK_+ ,
\]
which implies that $\bM$ has a finite number, $N$ say, of poles $\xi_j
\in \cK_+$. Here, $N$ depends on $\omega$. As $q_S(\xi_j) \in
\ii\R_+$, the poles are necessarily real. Moreover $\bM$ has no other
poles in $\cK_+$. We order the set of poles of the Weyl matrix,
\[
   \frac{\omega}{\sqrt{\mu_0}} < \xi_N < \xi_{N-1} < \cdots < \xi_1
\]
and invoke

\medskip\medskip

\begin{assumption} \label{as-simple}
The poles of $\bM$ in $\cK_+$ are simple.
\end{assumption}

\medskip\medskip

\begin{remark}
The poles of $\bM$ are among the zeros of the determinant of the Jost
function $\det\bF_\Theta$ or Rayleigh determinant on the physical
sheet, $\cK_+$, and correspond to the bound states or normal
modes. Assumption \ref{as-simple} is generically satisfied for all
frequencies $\omega \in\R_+$, that is, except possibly for a finite
number of frequencies. In the seismology literature, the Rayleigh
system is usually considered on a finite slab with traction-free
boundary conditions when simplicity of the normal modes is well known
\cite{WoodhouseDeuss2015}. In view of our outgoing radiation boundary
condition at one end represented by (\ref{outgoing}) below, this
result does not directly apply.
\end{remark}

\medskip\medskip

We associate ``energies'' with the mentioned poles, $\zeta_j = \xi_j^2
\in \Pi_+$ (cf.~(\ref{eq:Pi+})). We may introduce
$\widehat{\bM}=\widehat{\bM}(\zeta)$ by
\[
  \widehat{\bM}(\zeta(\xi)) = \bM(\xi),
\]
which thus has simple poles $\zeta_1,\ldots,\zeta_N$. We write
$\Lambda' = \{ \zeta_j \}_{j=1}^N$ and note that $\zeta_j >
\frac{\omega^2}{\mu_0}$.

\medskip\medskip

\begin{lemma} \label{lem:Mspecdat}
The matrix $\widehat{\bM}$ is analytic in $\Pi_+$ outside
$\Lambda'$. It admits the representation
\begin{equation} \label{(2.3.35)}
   \widehat{\bM}(\zeta) = \int_{-\infty}^{\frac{\omega^2}{\mu_0}}
   \frac{\widehat{\bT}(\eta)}{\zeta - \eta} \dd\eta
        + \sum_{j=1}^N \frac{\alpha_j}{\zeta - \zeta_j} ,
   \quad\zeta \in \Pi_+\setminus \Lambda' ,
\end{equation}
where
\begin{equation} \label{eq:alphj}
   \alpha_j = \Res_{\zeta = \zeta_j} \widehat{\bM}(\zeta)
   = \bF(0,\xi_j) u_j ,\quad u_j
   = 2 \xi_j \Res_{\xi = \xi_j} [\bF_{\Theta}(\xi)]^{-1}
\end{equation}
or
\begin{equation}
   \alpha_j =
   -[u^{\rm a}_j]^{\rm T} \int_0^\infty
         [\bF^{\rm a}(x,\xi_j)]^{\rm T} \bF(x,\xi) \dd x \, u_j ,
   \quad u^{\rm a}_j = 2 \xi_j
         \Res_{\xi = \xi_j} [\bF^{\rm a}_{\Theta}(\xi)]^{-1}
\end{equation}
or
\begin{equation} \label{important20}
   \alpha_j =
   \bF(0,\xi_j) \left(\bF_\Theta'(\xi_j) \right)^{-1}
            = - \ii\frac{ \mu_0}{\omega^2}
   \left[\left(\bF^{\rm a}_\Theta(-\xi_j)\right)^{\rm T}\right]^{-1}
   \ma q_P(\xi_j) & 0 \\ 0 & -q_S(\xi_j) \am
      \left(\bF_\Theta'(\xi_j) \right)^{-1}
\end{equation}
and $\widehat{\bT} = \widehat{\bT}(\zeta)$, $\widehat{\bT}(\zeta(\xi))
= \bT(\xi)$ with
\begin{equation} \label{eq:Tzet}
   \bT(\xi) = -\frac{\xi \mu_0}{\pi \omega^2}
   [(\bF_\Theta^{\rm a})^{\rm T}(-\xi)]^{-1}
   \ma q_P(\xi) & 0 \\ 0 & -q_S(\xi) \am [\bF_\Theta (\xi)]^{-1} ,
   \quad
   \zeta \in \left(-\infty,\frac{\omega^2}{\mu_0}\right] ,
\end{equation}
signifying the branch cut.
\end{lemma}

\begin{proof}
We fix a pole $\zeta_j = \xi_j^2$, use that
\[
  \bF(x,\xi) [\bF_{\Theta}(\xi)]^{-1} = \bS(x,\xi) +
                   \boldsymbol{\varphi}(x,\xi) \bM(\xi)
\]
(cf. (\ref{eq:defWsol}) and (\ref{eq:WvphiS})), and evaluate the
residue
\[
  \bF(x,\xi_j) \Res_{\xi = \xi_j} [\bF_{\Theta}(\xi)]^{-1}
       = \boldsymbol{\varphi}(x,\xi_j) \Res_{\xi = \xi_j} \bM(\xi) .
\]
As
\[
  \Res_{\zeta = \zeta_j} \widehat{\bM}(\zeta)
              = 2 \xi_j \Res_{\xi = \xi_j} \bM(\xi) =: \alpha_j
\]
we get
\begin{equation}
   \boldsymbol{\varphi}(x,\xi_j) \alpha_j = \bF(x,\xi_j) u_j ,\quad
   u_j = 2 \xi_j \Res_{\xi = \xi_j} [\bF_{\Theta}(\xi)]^{-1}
\end{equation}
In a similar fashion, we get
\begin{equation}
   \boldsymbol{\varphi}^{\rm a}(x,\xi_j) \alpha_j
        = [u^{\rm a}_j]^{\rm T} [\bF^{\rm a}(x,\xi_j)]^{\rm T} ,\quad
   u^{\rm a}_j = 2 \xi_j \Res_{\xi = \xi_j} [\bF^{\rm a}_{\Theta}(\xi)]^{-1}
\end{equation}
(cf. (\ref{eq:defWasol}), (\ref{eq:WavphiS}) and
(\ref{eq:MaM})). Integrating the derivative of the relevant Wronskian,
we obtain
\begin{equation}
   \lim_{x \to \infty} W(\bF^{\rm a}(x,\xi_j),\bF(x,\xi))
      - W(\bF^{\rm a}(x,\xi_j),\bF(x,\xi))|_{x=0}
   = (\xi_j^2 - \xi^2) \int_0^\infty
        [\bF^{\rm a}(x,\xi_j)]^{\rm T} \bF(x,\xi) \dd x
\end{equation}
for $\xi \in (\tfrac{\omega}{\sqrt{\mu_0}},\infty)$. Using the
asymptotics of $\bF_0(x,\xi)$ and $\bF^{\rm a}(x,\xi_j)$ as $x \to
\infty$, we get
\begin{equation}
   \lim_{\xi \to \xi_j} \frac{1}{\xi_j^2 - \xi^2} 
      \lim_{x \to \infty} W(\bF^{\rm a}(x,\xi_j),\bF(x,\xi)) = 0 .
\end{equation}
Hence,
\begin{equation}
   [u^{\rm a}_j]^{\rm T} \int_0^\infty
         [\bF^{\rm a}(x,\xi_j)]^{\rm T} \bF(x,\xi) \dd x \, u_j
   = -\alpha_j
\end{equation}
yielding a representation of $\alpha_j$ in terms of the Jost solutions.

We now prove (\ref{important20}). As the pole $\zeta_j=\xi_j^2$ is simple, we
also have
\begin{equation} \label{important3}
   \alpha_j
   = \bF(0,\xi_j) \left[\lim_{\xi \to \xi_j}
     (\xi^2 - \xi_j^2)^{-1} \bF_\Theta(\xi)\right]^{-1}
   = \frac{1}{2\xi_j}\bF(0,\xi_j) \left[\bF_\Theta'(\xi_j)\right]^{-1} .
\end{equation}
At $\xi = \xi_j$, we have
\[
   \bF'(0,\xi_j) = -\Theta(\xi_j) \bF(0,\xi_j) ,\quad
   \bF^{\rm a}_\Theta(\xi) = (\bF^{\rm a})'(0,\xi)
             + \Theta^{\rm T}(\xi) \bF^{\rm a}(0,\xi_j) .
\]   
Then using (\ref{wronskian-xi}) at $\xi = \xi_j$, we get
\begin{multline*}
   W(\bF^{\rm a}(0,-\xi_j) \bF(0,\xi_j))
   = \left[((\bF^{\rm a})')^{\rm T}(0,-\xi_j)
     + (\bF^{\rm a})^{\rm T} (0,-\xi_j) \Theta(\xi_j)\right]
     \bF(0,\xi_j)
\\
   = (\bF^{\rm a}_\Theta(-\xi_j))^{\rm T} \bF(0,\xi_j)
   = - 2 \ii \mu_0 \frac{\xi_j}{\omega^2}
         \ma q_P(\xi_j) & 0 \\ 0 & -q_S(\xi_j) \am
\end{multline*}
and using (\ref{important3}) we obtain (\ref{important20}).

The jump across the branch cut is obtained through
\begin{equation} \label{eq:hM+}
  \widehat{\bM}^+(\zeta) = \widehat{\bM}(\zeta + \ii 0)
      = \bM(\xi^*) = \bF(0,\xi^*) [\bF_{\Theta}(\xi^*)]^{-1} ,\quad
  \zeta \in \left(-\infty,\frac{\omega^2}{\mu_0}\right] ,
\end{equation}
where $\xi^* \in
[-\frac{\omega}{\sqrt{\mu_0}},\frac{\omega}{\sqrt{\mu_0}}] \cup \ii
\R$ approached from $\cK_+$, and similarly
\begin{equation} \label{eq:hM-}
  \widehat{\bM}^-(\zeta) = \widehat{\bM}(\zeta - \ii 0)
      = ([\bF^{\rm a}_{\Theta}(-\xi^*)]^{\rm T})^{-1}
            [\bF^{\rm a}(0,-\xi^*)]^{\rm T}  ,\quad
  \zeta \in \left(-\infty,\frac{\omega^2}{\mu_0}\right] .
\end{equation}
Using the definitions of the Jost functions, writing $\xi^* = \xi$, we
then get \begin{equation} \label{eq:TzetM+-}
   \widehat{\bT}(\zeta(\xi)) = \frac{1}{2\pi \ii}
   (\widehat{\bM}^+(\zeta(\xi)) - \widehat{\bM}^-(\zeta(\xi)))
    = \frac{1}{2 \pi \ii} ([\bF^{\rm a}_{\Theta}]^{\rm T}(-\xi))^{-1}
      W(\bF^{\rm a}(x,-\xi),\bF(x,\xi)) [\bF_{\Theta}(\xi)]^{-1} .
\end{equation}
With Lemma~\ref{lem:mwronskian}, we obtain (\ref{eq:Tzet}).
\end{proof}

\medskip\medskip

\noindent
In (\ref{(2.3.35)}) we distinguish, from a physics perspective, three
contributions:
\[
  \int_{-\infty}^0
           \frac{\widehat{\bT}(\eta)}{\zeta - \eta} \dd\eta
\]
from the evanescent modes,
\[
  \int_0^{\frac{\omega^2}{\mu_0}}
           \frac{\widehat{\bT}(\eta)}{\zeta - \eta} \dd\eta
\]
from the radiating modes, and
\[
  \sum_{j=1}^N \frac{\alpha_j}{\zeta - \zeta_j}
\]
from the guided modes.

\medskip\medskip

\begin{remark}\label{remJF}
We note that through (\ref{important20}), (\ref{eq:Tzet}) and using
(\ref{Jfunction-adjoint}), $\alpha_j$ and $\bT$ can be expressed in
terms of the Jost function only. Thus Lemma~\ref{lem:Mspecdat}
indicates that the Jost function encodes the boundary spectral data.
\end{remark}  

\subsection{Even extension}
\label{ssec:evenext}

In the original variable $\xi$, $\xi^2 = \zeta$, (\ref{(2.3.35)})
reads
\begin{align}
   \bM(\xi)& = \int_{-\infty}^{\frac{\omega^2}{\mu_0}}
     \frac{\widehat{\bT}(\eta)}{\xi^2 - \eta} \dd\eta
     + \sum_{j=1}^N \frac{\alpha_j}{\xi^2 - \xi_j^2}
\label{(2.3.35bis)} \\
   & = \int_{-\infty}^{\frac{\omega^2}{\mu_0}}
   \frac{\widehat{\bT}(\eta)}{\xi^2 - \eta} \dd\eta
   + \sum_{j=1}^N \frac{\alpha_j}{2\xi_j} \frac{1}{\xi - \xi_j}
   + \sum_{j=1}^N \frac{\alpha_j}{2\xi_j} \frac{(-1)}{\xi + \xi_j} ,
   \quad\xi \in \cK_+ \setminus \{\xi_j\}_{j=1}^N .
\nonumber
\end{align}
Using this representation, $M(\xi)$ has an artificial extension to
\[
   \C \setminus \bigcup\left\{\ii\R,
   \left[-\frac{\omega}{\sqrt{\mu_0}},
     \frac{\omega}{\sqrt{\mu_0}}\right] ,
   \{\pm \xi_j\}_{j=1}^N\right\}
   \equiv \cK \setminus \{\pm \xi_j\}_{j=1}^N
\]
as an even function
\[
   \bM(-\xi) = \bM(\xi) ,
\]
which we will employ in the further analysis. We emphasize that this
extension is fundamentally different from the above mentioned
meromorphic continuation.

\medskip\medskip

\noindent
In the further analysis we invoke

\medskip\medskip

\begin{assumption} \label{ass:2}
The parameter functions, $\lambda$ and $\mu$, are such that there is
no pole of $\bM(\xi)$ with $\Im q_S = 0$ except, possibly, at $\xi =
\frac{\omega}{\sqrt{\mu_0}}$ as a one-sided limit in $\cK_+$.
\end{assumption}

\subsection{Data for the original Rayleigh system}

\subsubsection{Weyl matrix}

In Appendix~\ref{app:ND} we develop a relation between the
Neumann-to-Dirichlet map ($\bND$) of the original Rayleigh system and
the Weyl matrix induced by the Markushevich substitutions. From
\[
   \bND(\xi) = \left[\ma \displaystyle{
   \ii \frac{\mu_0}{\mu(0)}} & 0 \\ 0 & 0 \am
     + \ma 0 & \displaystyle{\frac12 \ii} \\ \displaystyle{
       -\frac{\mu_0}{\mu(0)} \xi} & 0 \am
       \bM(\xi)\right]
           \ma \displaystyle{\frac{1}{2 \mu_0 \xi}} & 0 \\
           \displaystyle{\frac{\mu'(0)}{\mu^2(0)} \frac{1}{\xi}} &
           \displaystyle{\frac{\ii}{\mu(0)}} \am
\]  
(cf.~(\ref{eq:C-NDM})) it follows that $\bND$ and $\bM$ have the same
poles and their jumps across branch cuts are explicitly related; this
relationship depends on $\mu(0)$, $\mu'(0)$ and $\mu_0$.

\subsubsection{Jost solution}

In addition to the Weyl matrix, we need the Jost solution at $x = 0$
as the data. We let $\mathbf{w} = [w_P \,\, w_S]$ denote the Jost
solution of the Rayleigh system before the Markushevich transform,
that is, both columns of $\mathbf{w}$ satisfy
(\ref{1rw(1.1)})-(\ref{1rw(1.2)}) and conditions
\[
   \mathbf{w} = [w^+_{P,0} \,\, w_{S,0}^+]\quad\text{for}\quad
   x \geq H ,
\]
where the ``reference" Jost solution comprised of $w_{P,0}+$ and
$w_{S,0}^+$ is given in (\ref{1rw(8.5)new})-(\ref{1rw(8.4)new}). The
Jost solution can be excited by imposing the ``outgoing radiation''
conditions at the bottom of the slab
\begin{equation}\label{outgoing}\mathbf{w}'(H,.) - \ii \, \mathbf{w}(H,.) \ma q_P & 0 \\ 0 & q_S
\am = 0\end{equation} supplemented with Dirichlet boundary value,
\[
   \mathbf{w}(H,\xi) = [w_P(H,\xi) \,\, w_S(H,\xi)]
                     = [w_{P,0}^+(H,\xi) \,\, w_{S,0}^+(H,\xi)]
\]
at the bottom of the slab, $x = H$, and observed at $x = 0$, giving
$\mathbf{w}(0,\xi)$. Upon the Markushevich substitution, this yields
$\bF(0,\xi)$.

\subsubsection{Jost function} \label{sJf}

 The Jost functions $\bF_\Theta(\xi)$ and $\bF_\Theta^{\rm
    a}(\xi)$ can be considered as alternative data. By
  (\ref{eq:JfB})-(\ref{eq:JafB}) these are directly related to the
  boundary matrix of the original Rayleigh problem
  (cf.~(\ref{eq:Bwf}))
\[
   \bB(\mathbf{w}) = \ma b_-(w_P) & b_-(w_S) \\
                         a_-(w_P) & a_-(w_S) \am ,
\]
where
\begin{equation}
   \mathbf{w} = [w_P \,\, w_S]
\end{equation}
is the Jost solution discussed above. By (\ref{Jfunction-adjoint}),
the Jost function determines the adjoint Jost function if $\mu(0)$,
$\mu'(0)$ and $\mu_0$ are known.

\subsection{Unique recovery of a potential of Lam\'{e} type}

In the following lemma, proposition and theorem we assume
  that $H$, $\lambda_0$, $\mu_0$, $\mu(0)$ and $\mu'(0)$ are known.
We introduce the expansion of the Jost solution at the boundary,
\begin{equation} \label{eq:F0xi}
   \bF(0,\xi) = \xi \bG_0(0,\xi) + \bG_1(0) + \bR(\xi) ,\quad
   \bR(\xi) = \mathcal{O}\left(\frac{1}{|\xi|}\right) .
\end{equation}
In Theorem~\ref{th-Jost-solutions}, we will construct explicit
expressions for $\bG_0(0,\xi)$ and $\bG_1(0,\xi)$.

\medskip\medskip

\begin{lemma}
Given $\lambda_0$ and $\mu_0$. The mapping from $G^H$
(cf. (\ref{1rw(3.6)new})) to $(\bG_0(0,\xi), \bG_1(0,\xi))$ for any
pair of frequencies, $\omega_1 \neq \omega_2 \in \R_+$, is an injection.
\end{lemma}

\medskip\medskip

\noindent
The proof is given in Subsection~\ref{ssec:7.2}. Thus, $(\bG_0(0,\xi),
\bG_1(0,\xi))$ for any two frequencies $\omega_1 \neq \omega_2 \in\R_+$
determine $G^H$. Moreover, $G^H$ together with $H$, $\lambda_0$,
$\mu_0$ and $\omega$ determine $Q_0$.

\medskip\medskip

\begin{proposition} \label{prop:V}
Given $G^H$. For $\omega$ fixed, let $V_1, V_2$ be compactly supported
on $[0,H]$ and belong to $L^1([0,H])$ with associated Weyl matrices
$\bM_1$, $\bM_2$. If $H$, $\lambda_0$, $\mu_0$, $\mu(0)$ and $\mu'(0)$
are known and Assumptions~\ref{as-simple} and \ref{ass:2} hold true,
then $\bM_2(\xi) = \bM_1(\xi)$ for all $\xi \in \cK_+$ implies that
$V_2 = V_1$.
\end{proposition}

\medskip\medskip

\noindent
The proof is given in Subsection~\ref{ssec:7.1}. Thus, $G^H$ together
with $\bM(\xi)$ determine $V$. By implication, $(\bG_0(0,\xi),
\bG_1(0,\xi))$ for any two frequencies $\omega_1 \neq \omega_2 \in \R_+$
and $\bM(\xi)$ determine $Q$. Furthermore, from a Lam\'{e}-type $Q$
for any pair of frequencies, $\omega_1 \neq \omega_2 \in \R_+$, we
recover $\lambda$ and $\mu$, which is proved in
Subsection~\ref{ssec:7.3}.

\medskip\medskip

\noindent
We need both the Weyl function, or $\bND$ map, and the Jost solution
at the boundary for the unique recovery of Lam\'{e}
parameters. Alternatively, we may use the Jost function
$\bF_\Theta(\xi)$ as the data, as by Lemma~\ref{lem:Mspecdat} and
Remark~\ref{remJF}, assuming that $\lambda_0$ and $\mu_0$ are known,
$\bF_\Theta(\xi)$ determines the Weyl function $\bM$ and, by
(\ref{eq:Mdef}), $\bF(0,\xi)$. We recall that $\bF_\Theta(\xi)$ also
determines $\bF^{\rm a}_\Theta(\xi)$. We obtain

\medskip\medskip

\begin{theorem} \label{thm:I}
Let $Q_1$, $Q_2$ be of Lam\'{e} type with associated Jost functions
$\bF_{\Theta;1}$, $\bF_{\Theta;2}$. Assume that $H$, $\lambda_0$,
$\mu_0$, $\mu(0)$ and $\mu'(0)$ are known. Then $\bF_{\Theta;2}(\xi) =
\bF_{\Theta;1}(\xi)$ for all $\xi \in \cK_+$ and any pair of
frequencies, $\omega_1 \neq \omega_2 \in \R_+$, subject to
Assumptions~\ref{as-simple} and \ref{ass:2}, implies that $Q_2 = Q_1$.
\end{theorem}

\section{Asymptotic expansions}

\subsection{Jost solutions}

In this subsection, we establish fundamental properties of the Jost
solutions. We refer to Appendix~\ref{app:GF} for a representation of
the Green's function $\bcG = \bcG(x,y)$ introduced in
(\ref{StLunpert_bis}).


\medskip\medskip

\begin{theorem} \label{th-Jost-solutions}
For any fixed $x \ge 0$, the Jost solution, $\bF$, is analytic in
$\xi$ on $\cK_+$, of exponential type, and satisfies
\begin{align*}
  \bF(x,\xi) &= \bF_0(x,\xi)
  \overbrace{- \int_x^H \bcG(x,y) V(y) \bF_0(y,\xi) \dd y}^{
  \bF_1(x,\xi)}
                              + \sum_{k=2}^\infty \bF_k(x,\xi) ,
\\
  \bF_k(x,\xi) &= \frac{|\xi|}{k!}
    \mathcal{O}\left(\frac{1}{\max\{|\xi|,1\}}\right)^k
    e^{-x \Im q_S(\xi)} .
\end{align*}
As $|\xi| \to \infty$, $\xi \in \cK_+$,
\begin{equation} \label{xilimitJost}
  \bF(x,\xi) = e^{-x \xi} \xi\left(
     \bG_0(x,\xi) + \frac{1}{\xi} \bG_1(x)
         + \mathcal{O}\left(\frac{1}{|\xi|^2}\right)\right) ,
\end{equation}
where
\begin{equation} \label{xilimitJost-G0}
  \bG_0(x,\xi) = -\frac{\mu_0}{\omega^2} \ma G_{11}^H & G_{11}^H \\ \\
                 G_{12}^H & G_{12}^H \am
        + \frac{1}{\xi} \ma \displaystyle
    \frac{c_0}{2}G_{11}^H H + G_{21}^H & -\frac12 x G_{11}^H \\ \\
    \displaystyle
    \frac{c_0}{2}G_{12}^H H + G_{22}^H &  -\frac12 x G_{12}^H \am
\end{equation}
and
\begin{equation} \label{xilimitJost-G1}
   \bG_1(x) = -\frac{1}{2} \frac{\mu_0}{\omega^2}
     \int_x^H V(y) \dd y
     \ma G_{11}^H & G_{11}^H \\[0.25cm]
                              G_{12}^H & G_{12}^H \am .
\end{equation}
Analogous properties and an expansion can be obtained for $\bF^{\rm
  a}$.
\end{theorem}

\medskip\medskip 

\begin{remark} 
The proof of Theorem~\ref{th-Jost-solutions} is based on an iteration
of Volterra-type equation (\ref{Volterra1}) (and (\ref{Volterra2}))
following a standard argument \cite{Marchenko1986}. We note that the
very reason to perform the Markushevich transform was to re-write the
Rayleigh problem in Schr{\"o}dinger form and then as a Volterra-type
integral equation with bounded kernel.
\end{remark}

\medskip\medskip

\noindent
In the above
theorem, \textit{$\bG_0$ contains a contribution
  $\mathcal{O}(\frac{1}{|\xi|})$; this contribution is essential to
  ensure that $\bG_0(0,\xi)$ is invertible while $\bG_0$ only depends
  on $Q_0$}. We write
\begin{equation} 
   \bF(0,\xi) = \xi \bG_0(0,\xi) + \bG_1(0) + \bR(\xi) ,\quad
   \bR(\xi) = \mathcal{O}\left(\frac{1}{|\xi|}\right) .
\end{equation}

Furthermore, the Jost function, $\bF_\Theta$, and $\det \bF_\Theta$
are analytic in $\xi$ where $\Im q_P > 0 ,\ \Im q_S > 0$ ($\xi \in
\cK_+$, see Appendix~\ref{app:RS}) and continuous in $\xi$ where $\Im
q_P \ge 0 ,\ q_P \ne 0$, $\Im q_S \ge 0 ,\ q_S \ne 0$. We obtain

\medskip\medskip

\begin{corollary} \label{th-Jost function}
The Jost function admits the asymptotic expansion, as $|\xi| \to
\infty$, $\xi \in \cK_+$,
\begin{align}
   \bF_\Theta(\xi) =& -2 \xi^3 \frac{\mu_0^2}{\omega^2\mu(0)}
        G_{11}^H \ma 0 & 0 \\[0.25cm] 1 & 1 \am
        +\xi^2 \left(\mstrut{0.55cm}\right.
        \frac{\mu_0}{\omega^2}
        \ma G_{11}^H & G_{11}^H \\[0.25cm] G_{12}^H & G_{12}^H
        \am
\nonumber\\
   & + 2 \frac{\mu_0}{\mu(0)} \left( G_{11}^H \frac12 c_0 H
      + G_{21}^H \right) \ma 0 & 0 \\ 1 & 0 \am\label{asJostF}
\\
   &\hspace*{3.5cm}
   - \frac{1}{2} \frac{\mu_0}{\omega^2}
     \ma 0 & 0 \\ \displaystyle{2 \frac{\mu_0}{\mu(0)}} & 0 \am 
     \int_0^H V(y) \dd y
     \ma G_{11}^H & G_{11}^H \\[0.25cm] G_{12}^H & G_{12}^H \am
     \left.\mstrut{0.55cm}\right)
     + o(|\xi|^2) .
\nonumber
\end{align}  
\end{corollary}

\medskip\medskip

\begin{corollary} \label{th-varphi solution}
Solutions $\bS(x,\xi)$ and $\boldsymbol{\varphi}(x,\xi)$
(cf.~(\ref{eq:WvphiS}) and above) are entire on $\C$ and even
functions of $\xi$ of exponential type. For the first and second
derivatives, the following estimate holds true
\begin{equation} \label{boundvarphi}
   \|\boldsymbol{\varphi}^{(k)}(x,\xi)\| \leq C| \xi|^{k+1}
               e^{|\Re \xi| x} ,\quad \xi \in \cK_+ ,\quad k=0,1
\end{equation}
uniformly in $x \ge 0$.
\end{corollary}

\subsection{Weyl matrix}

Next, we study the Weyl matrix introduced in (\ref{eq:Mdef}), and
derive its asymptotic expansion as $|\xi| \to \infty$, $\xi \in
\cK_+$. It appears that it is more convenient to start the analysis
with its inverse,
\begin{equation}
  \bM^{-1}(\xi) = \bF'(0,\xi) \bF^{-1}(0,\xi) + \Theta(\xi) ,
\end{equation}
where $\Theta$ is given in (\ref{1rw(2.15)new}). The following result
is an analogue of \cite[Proposition~2]{Bealsetal1995}.

\medskip\medskip

\begin{lemma} \label{lem:Wasymp}
The Weyl matrix, $\bM(\xi)$, and its inverse $\bM(\xi)^{-1}$, have the
following asymptotic expansions for $|\xi| \to \infty$, $\xi \in
\cK_+$,
\begin{enumerate}
\item 
If $Q \in C^\infty$ and all its derivatives are integrable on $\R_+$,
then $\bM(\xi)^{-1}$ has an asymptotic expansion to all orders,
\begin{align}
  \bM(\xi)^{-1} &= \Theta(\xi) + \xi \sum_{k=0}^\infty \xi^{-k} X_k(0)
\label{(2.7)}\\
  X_0 = -&I_2 ,\quad X_1 = 0 ,\quad X_2= -\frac12 Q ,\quad
  X_3 = -\frac{1}{4} Q' ,\quad
  2 X_{k+1} = X_k' + \sum_{j=1}^k X_j X_{4-j}
\nonumber
\end{align}
(cf.~(\ref{1rw(2.15)new})).
\item 
If $Q \in C^\infty$ and all its derivatives are integrable on $\R_+$,
then $\bM(\xi)$ has an asymptotic expansion to all orders, $\bM(\xi) =
\sum_{k=0}^\infty \xi^{-k} Y_k$. Weakening the condition, if $Q \in
C^3$ and its first three derivatives are integrable on $\R_+$,
\begin{equation} \label{(2.4)} 
   \bM(\xi) = Y_0 + \xi^{-1} Y_1 + \xi^{-2} Y_2
            + \xi^{-3} \left(Y_3 + \boldsymbol{\alpha}(\xi)\right) ,
\end{equation}
where $\boldsymbol{\alpha}(\xi) \to 0$ as $|\xi| \to \infty$, $\xi \in
\cK_+$, with
\begin{align}
  Y_0 &= \frac{1}{1 - 2 \varpi \theta_2} \ma 0 & 0 \\ y_{0;21} & 0 \am  
\\
  Y_1 &= \frac{1}{1 - 2 \varpi \theta_2}
                      \ma - 1 & 0 \\ y_{1;21} & -1 \am ,
\\
  Y_2 &= \frac{1}{1 - 2 \varpi \theta_2}
         \ma y_{2;11} & y_{2;12} \\
             y_{2;21} & y_{2;22} \am ,
\end{align}
in which (cf.~(\ref{1rw(2.15)new})) ${\displaystyle 1 - 2 \varpi
  \theta_2 = \frac{\lambda(0) + \mu(0)}{\lambda(0) + 2 \mu(0)}}$ and
\begin{align*}
  y_{0;21} &= -2 \varpi ,
  \\[0.25cm]  
  y_{1;21} &= \frac{(1 - \varpi) Q_{12}(0)}{
             (1 - 2 \varpi \theta_2) \theta_2}
        + \frac{2 \varpi (\theta_3 + \varpi Q_{12}(0))}{
              1 - 2 \varpi \theta_2} ,
  \\
  y_{2;11} &= \frac{1}{2} y_{1;21} ,
  \\[0.25cm]
  y_{2;12} &= - \theta_2 ,
  \\[0.25cm]
  y_{2;21} &= \frac{Q_{12}'(0) - y_{1;21} Q_{12}(0) + Q_{11}(0)
      + Q_{22}(0) - y_{1;21} \theta_3 + \theta_1}{1 - 2 \theta_2} ,
\\
  y_{2;22} &= \frac{\varpi Q_{12}(0) + 2 \varpi \theta_2 \theta_3}{
              1 - 2 \varpi \theta_2} .
\end{align*}
%
%
\end{enumerate}
\end{lemma}

\begin{proof}
We first prove (a). Where $\bF(x,\xi)$ is invertible ($\xi \ne 0$)
(\ref{StLunpert_matrix}) is equivalent to
\begin{equation} \label{eq:FpF-1}
   \frac{\dd}{\dd x} (\bF'(x,\xi) \bF(x,\xi)^{-1})
        + (\bF'(x,\xi) \bF(x,\xi)^{-1})^2 = Q(x) + \xi^2 .
\end{equation}
Using analytic properties of the Jost solution, it follows that
$\bF_0(x,\xi)^{-1} \bF(x,\xi)$ (cf.~(\ref{eq:Jostsoldef})) admits an
expansion to all orders of $\xi^{-1}$ and that this expansion can be
differentiated term by term. Such an expansion also exists for
$\bF(x,\xi)$ and, hence, for $\bF'(x,\xi)$. Thus the following
expansion exists,
\[
  \bF'(x,\xi) \bF(x,\xi)^{-1} = \xi \sum_{k=0}^{\infty} \xi^{-k} X_k(x) .
\]
We insert this expansion in (\ref{eq:FpF-1}) and note that
\[
  \bM(\xi)^{-1} = \Theta(\xi) + \bF'(x,\xi) \bF(x,\xi)^{-1} |_{x = 0} . 
\]
We prove (b) using (a) and $\bM(\xi) \bM(\xi)^{-1} = I_2$ by explicit
calculations.
\end{proof}

\medskip\medskip

Capturing the leading orders, we introduce
\begin{equation} \label{eq:tilY}
  \widetilde{Y}(\xi) = Y_0 + \frac{1}{\xi} Y_1
       = \frac{1}{\xi (1 - 2 \varpi \theta_2)}
       \ma -1 & 0 \\ 2 \left(b - \varpi \xi\right) & -1 \am ,\quad
       b = \frac12 y_{1;21} ,
\end{equation}
with inverse
\[
  \widetilde{Y}(\xi)^{-1} = \xi^2 (1 - 2 \varpi \theta_2)
          \ma 0 & 0 \\ 2 \varpi & 0 \am
          + \xi (1 - 2 \varpi \theta_2) \ma -1 & 0 \\ -2b & -1 \am
\]
and the properties
\begin{align}
   Y_0 \widetilde{Y}(\xi)^{-1} &=
       \xi \ma 0 & 0 \\ 2 \varpi & 0 \am ,
\\
   Y_1 \widetilde{Y}(\xi)^{-1} &= \widetilde{Y}(\xi)^{-1} Y_1
       = \xi \ma 1 & 0 \\ -2 \varpi \xi & 1 \am = \xi E(-\xi) ,
\label{eq:Y1tYD}
\end{align}
where
\begin{equation} \label{eq:Edef}
   E(\xi) = \ma 1 & 0 \\ 2 \varpi \xi & 1 \am
\end{equation}
forms a group as $E(\xi)^{-1} = E(-\xi)$ and
\[
   E(\xi_1) E(\xi_2)^{-1} = \ma 1 & 0 \\
                           2 \varpi (\xi_1 - \xi_2) & 1 \am .
\]
Then
\begin{equation} \label{eq:MT01-1}
  \bM(\xi) \widetilde{Y}(\xi)^{-1} = T_0 - \frac{1}{\xi} T_1
             + \mathcal{O}\left(\frac{1}{|\xi|^2}\right) ,\quad
  \xi \in \cK_+ ,
\end{equation}
where
\begin{equation} \label{eq:MT01-2}
  T_0 = \ma 1 - 2 \varpi \theta_1 & 0 \\
             2 \varpi (b - \theta_3) & 1 \am . 
\end{equation}

\medskip\medskip

\noindent
From (\ref{(2.4)}) and (\ref{eq:TzetM+-}) with (\ref{eq:hM+}),
  (\ref{eq:hM-}), it follows that
\begin{equation} \label{2.1.79}
  \bT(\zeta) = \frac{1}{\pi \ii \xi}
  \left(Y_1 + \frac{1}{\xi^2} Y_3
            + o\left(\frac{1}{|\xi|^2}\right)\right) ,\quad
  \zeta \in \left(-\infty,\frac{\omega^2}{\mu_0}\right]
\end{equation}
(cf~(\ref{eq:Tzet})), where $\xi = \sqrt{\zeta}$ is defined below
(\ref{Mextension}), while on the branch cut if $\zeta < 0$ then $\xi
\in \ii \R$.

\section{Gel'fand-Levitan type equation and proof of
  Theorem~\ref{thm:I}}
\label{sec:GL}

Using the results from the previous section, we obtain an asymptotic
expansion for $\bF(x,\xi) [\bF(0,\xi)]^{-1}$ in

\medskip\medskip

\begin{lemma} \label{L-FF-1}
The following asymptotic expansion holds true 
\[
   \bF(x,\xi) [\bF(0,\xi)]^{-1} = e^{-x \xi} \left(I_2
      + \frac{1}{\xi} \bD(x)
           + o\left(\frac{1}{|\xi|}\right) \right) ,
   \quad \xi \in \cK_+ ,
\]
where
\begin{equation*}
   \bD(x) = \frac12 \int_0^x V(y) \dd y \ma
   \displaystyle{-G_{11}^H \left(\frac{c_0}{2} G_{12}^H H
     + G_{22}^H \right)} &
   \displaystyle{ G_{11}^H \left(\frac{c_0}{2} G_{11}^H H
     + G_{21}^H \right)} \\[0.15cm]
   \displaystyle{-G_{12}^H \left(\frac{c_0}{2} G_{12}^H H
     + G_{22}^H \right)} &
   \displaystyle{ G_{12}^H \left(\frac{c_0}{2} G_{11}^H H
     + G_{21}^H \right)} \am .
\end{equation*}
\end{lemma}

\begin{proof}
The statement follows from an explicit calculation using Theorem
~\ref{th-Jost-solutions}, that is,
(\ref{xilimitJost})-(\ref{xilimitJost-G1}), and (\ref{eq:F0xi}).
\end{proof}

\medskip\medskip

\noindent
From (\ref{eq:PhiM}) upon employing Lemma~\ref{lem:Wasymp}, that is,
(\ref{(2.4)}), and using
\begin{equation} \label{(3.5)}
   \boldsymbol{\Phi}(x,\pm \xi) = \bF(x,\pm \xi) [\bF(0,\pm \xi)]^{-1}
           \bM(\pm \xi) ,\quad \pm \xi \in \cK_+
\end{equation}
(and extension to the branch cuts) we obtain

\medskip\medskip

\begin{corollary} \label{cor:Phipmasym}
The following asymptotic expansion holds true
$$
   \boldsymbol{\Phi}(x,\pm \xi) = e^{-\xi x} \left(
     Y_0 \pm \frac{1}{\xi} \left(Y_1
   + \left(\bD(x) \mp \frac{\omega^2}{2 \mu_0} x \right)
     Y_0 \right)
   + o\left(\frac{1}{|\xi|}\right)\right) ,
   \quad \pm \xi \in \cK_+ .
$$
\end{corollary}

\medskip\medskip

\noindent
This corollary immediately implies that
\begin{equation} \label{eq:Phipmderasym}
   \boldsymbol{\Phi}'(x,\pm \xi) = -\xi e^{-\xi x} \left(
     Y_0 \pm \frac{1}{\xi} \left( Y_1
   + \left(\bD(x) \mp \frac{\omega^2}{2 \mu_0} x \right)
     Y_0 \right)
   + o\left(\frac{1}{|\xi|}\right)\right) ,
   \quad \pm \xi \in \cK_+ .
\end{equation}

\subsection{Recovery of $V$ assuming that $G^H$ is known}
\label{ssec:7.1}

For the proof of Theorem~\ref{thm:I}, we change variables through the
transformation (Appendix~\ref{app:RS})
\begin{equation} \label{eq:kxi}
   \xi \to k ,\quad
   k = \sqrt{\frac{\omega^2}{\mu_0} - \xi^2} ,\quad
   \xi = -\ii k + {\mathcal O}\left(\frac{1}{|k|}\right),\quad \xi\in \cK \quad
   \text{($\omega$ fixed)}.
\end{equation}
The choice of sign is determined by letting $\xi \in \cK_+$ (where
also $\Im q_P(\xi) > 0$) correspond to $k \in \C_+$ ($\Im k >
0$). The inverse of the transformation is defined on
  $\C_+$ and written as $\xi(k)$. The branch cut in $\xi \in \cK_S$
  corresponds with $\Im k = 0$. We let
\[
   \cK_+ \to \cK_- :\ \xi \to -\xi\quad\text{correspond with}\quad
   \C_+ \to \C_- :\ k \to -k ,
\]
where $\C_+$ ($\Im k < 0$). First, we give some basic asymptotic
expansions. To next order, we have 
\begin{equation*}
  \ii k + \xi
    = \frac{\omega^2}{2 \xi \mu_0}
  + {\mathcal O}\left(\frac{1}{|\xi|^2}\right),\quad\quad \xi\in \cK_S
\end{equation*}
so that
\begin{equation*}
  e^{\mp (\ii k + \xi) x} = 1 \mp \frac{\omega^2}{2 \xi \mu_0} x
  + {\mathcal O}\left(\frac{1}{|\xi|^2}\right) ,\quad
  e^{\mp \ii k x} = e^{\pm \xi x}
  \left(1 \mp \frac{\omega^2}{2 \xi \mu_0} x
     + {\mathcal O}\left(\frac{1}{|\xi|^2}\right)\right)\quad
  \text{as}\ |\xi| \to \infty .
\end{equation*}
 We introduce
\begin{equation}
  \bM_\pm(k) = \bM(\pm \xi) ,\quad
  \pm \xi \in \cK_+ ,\quad \pm \Im k \geq 0 ,
\end{equation}
and two more solutions (cf.~(\ref{eq:WvphiS}))
\begin{equation} \label{eq:Phi+-Svarphi}
   \boldsymbol{\Phi}_\pm(x,k) = \boldsymbol{\Phi}(x,\pm \xi)
   = \widetilde\bS(x,k)
     + \widetilde{\boldsymbol{\varphi}}(x,k) \bM_\pm(k) ,\quad
     \pm\xi \in \cK_+ ,\quad \pm \Im k \geq 0 ,
\end{equation}
where we identify
\begin{equation}
   \widetilde\bS(x,k) = \bS(x,\xi(k)) ,\quad
\quad
   \widetilde{\boldsymbol{\varphi}}(x,k)
                           = \boldsymbol{\varphi}(x,\xi(k))
\end{equation}
with
\begin{equation}
   \widetilde\Theta(k) = \Theta(\xi(k)) .
\end{equation}
$\widetilde\bS(x,k)$, $\widetilde{\boldsymbol{\varphi}}(x,k)$ and
$\widetilde\Theta(k)$ are entire and even in $k$ as solutions and by
the relevant boundary conditions. While $\bM_+$ is
  defined for $k \in \C_+$ ($\Im k > 0$), $\bM_-$ is defined for $k
  \in \C_-$ ($\Im k < 0$). It follows that for \textit{real-valued}
$k$
\begin{equation} \label{(3.7)}
   \boldsymbol{\Phi}_+(x,k) - \boldsymbol{\Phi}_-(x,k)
   = \widetilde{\boldsymbol{\varphi}}(x,k) (\bM_+(k) - \bM_-(k)) .
\end{equation}
Applying Lemma~\ref{lem:Wasymp} and using (\ref{eq:tilY})), we find
that
\begin{equation} \label{(3.1)}
   \frac{2}{\xi(k)} Y_1
   = \widetilde{Y}(\xi(k)) - \widetilde{Y}(-\xi(k))
   = \bM(\xi(k)) - \bM(-\xi(k))
                 + {\mathcal O}\left(\frac{1}{|k|^3}\right) .
\end{equation}
In the later analysis, for real-valued $k$, we will employ the
notation $\widetilde{Y}(-\ii k)$ and $E(-\ii k)$, substituting $-\ii
k$ for $\xi$ (or $\xi(k)$). By abuse of notation, in this subsection,
we will omit the $\widetilde{\phantom{o}}$ in $\widetilde\bS$,
$\widetilde{\boldsymbol{\varphi}}$ and $\widetilde\Theta$.

Clearly, the matrix functions $\boldsymbol{\Phi}_\pm$ satisfy the boundary
conditions (cf.~(\ref{eq:PhiThe0-1}))
\[
   \boldsymbol{\Phi}_\pm(0,k) = \bM_\pm (k) ,\quad
   \boldsymbol{\Phi}_\pm'(0,k) + \Theta(k) \boldsymbol{\Phi}_\pm (0,k) = I_2 ,
   \quad \pm \Im k > 0 .
\]
%
%
%
%
%
Using Corollary~\ref{cor:Phipmasym} and (\ref{(3.7)}), for real-valued
$k$, we note that
\begin{equation} \label{Phi-Phi}
   \boldsymbol{\Phi}_+(x,k) - \boldsymbol{\Phi}_-(x,k)
   = \boldsymbol{\varphi}(x,k) \left(\frac{2}{-\ii k} Y_1
     + \mathcal{O}\left(\frac{1}{|k|^3}\right)\right)
\end{equation}
and the expansions in Corollary~\ref{cor:Phipmasym} imply

\medskip\medskip

\begin{lemma} \label{L-leading_varphi}
For $k \in \R$, the following asymptotic expansion holds true
\begin{equation} \label{exp_varphi}
   \boldsymbol{\varphi}(x,k) = \boldsymbol{\phi}_0(x,k)
     + (e^{\ii k x} + e^{-\ii k x}) \bm1(x)
            + (e^{\ii k x} - e^{-\ii k  x})
                         \mathcal{O}\left(\frac{1}{|k|}\right) ,
\end{equation}
where
\begin{equation} \label{eq:phi0def}
   \boldsymbol{\phi}_0(x,k)
         = -\frac{1}{2} \ii k (e^{\ii k x} - e^{-\ii k x})
                 Y_0 Y_1^{-1} ,
\end{equation}
in which
\[
\frac12Y_0 Y_1^{-1} = \varpi (1 - 2\varpi \theta_2) \ma
 0
  & 0 \\
    1 & 0 \am
\]
and $\bm1(x)$ is independent of $k$.
\end{lemma}

\medskip\medskip

\noindent
The lemma above implies that for $k \in \R$,
\[
   \boldsymbol{\varphi}(x,k) \bM_{\pm}(k) = \mathcal{O}(1) 
\]
and we have
\begin{equation} \label{exp_varphi_bis}
   e^{-|k| x} \boldsymbol{\varphi}(x,k)
     =  \ii k \frac12  Y_0 Y_1^{-1} + \mathcal{O}(1) .
\end{equation}
We introduce
\begin{equation} \label{eq:hatvphi}
   \widehat{\boldsymbol{\varphi}}(x,k)
             = \boldsymbol{\varphi}(x,k) - \boldsymbol{\phi}_0(x,k)
\end{equation}
(cf.~(\ref{eq:phi0def})) for the later analysis.

\medskip\medskip

\begin{definition}
We let
\[
  \Psi(x,k) = \left\{\begin{array}{lc} \displaystyle
  \Psi_+(x,k) = e^{\ii k x} \left(\boldsymbol{\Phi}_+(x,k)
      - \frac{2 \ii}{k} \boldsymbol{\varphi}(x,k) Y_1\right)
      \widetilde{Y}^{-1}(\ii k) ,& k \in \C_+ ,
  \\ \\
  \Psi_-(x,k) = e^{\ii k x} \boldsymbol{\Phi}_-(x,k)
      \widetilde{Y}^{-1}(\ii k) ,& k \in \C_-
  \end{array}\right.
\]
(cf.~(\ref{eq:tilY})).
\end{definition}

\medskip\medskip

\noindent
The matrix function $\Psi$ is meromorphic and defined on $\C$ through
the above mentioned extension of $\boldsymbol{\Phi}$ to branch cuts, which is
important for the later contour integration in the proof of
Proposition~\ref{prop:Gelfand-Levitan}. However, as elucidated in
Subsection~\ref{ssec:evenext}, throughout we intrinsically use
functions defined on the physical sheet only.

From (\ref{(3.5)}) it follows that $\Psi$ inherits its poles from
$\bM$ (see Subsection~\ref{ssec:CauchyI})

\medskip\medskip

\begin{lemma} \label{lem:Psipoles}
The matrix function $\Psi$ has poles at $\pm k_j$ with
\[
  k_j = \sqrt{\frac{\omega^2}{\mu_0} - \xi_j^2} .
\]
The residues are given by
\begin{equation} \label{residues}
  \Res_{k = \pm k_j} \Psi(x,k)
  = \frac{1}{2 \ii} e^{\pm \ii k_j x}
            \boldsymbol{\varphi}(x,k_j) C_j E(\mp \ii k_j) ,
  \quad C_j = -2 k_j B_j Y_1^{-1} = \alpha_j Y_1^{-1} .
\end{equation}
\end{lemma}

\begin{proof}
We have (cf.~(\ref{eq:Phi+-Svarphi}))
\[
  \Psi(x,k) = e^{\ii k x} \left(\bS(x,k) Y_1^{-1}
      + \boldsymbol{\varphi}(x,k) \bM_+(k) Y_1^{-1} - \frac{2 \ii}{k}
        \boldsymbol{\varphi}(x,k)\right) Y_1 \widetilde{Y}(\ii k)^{-1}
\]
for $\Im k > 0$, where $\bS$ and $\boldsymbol{\varphi}$ are entire in
$k$. As
\begin{multline} \label{eq:alphjBj}
  \alpha_j = \lim_{\xi \to \xi_j} (\xi^2 - \xi_j^2) \bM(\xi)
    = -\lim_{k \to k_j} (k^2 - k_j^2) \bM_+(k)
\\
    = -2 k_j \lim_{k \to k_j} (k - k_j) \bM_+(k)
    = -2 k_j B_j ,\quad
    B_j = \Res_{k = k_j} \bM_+(k) ,
\end{multline}
we obtain
\[
  \Res_{k = k_j} \Psi(x,k) = e^{\ii k_j x}
  \boldsymbol{\varphi}(x,k_j) B_j Y_1^{-1}
        Y_1 \widetilde{Y}(\ii k_j)^{-1} ,
\]
which, with (\ref{eq:Y1tYD}), implies the statement. Here, we use that
the poles of $\bM$ are simple.
\end{proof}

\medskip\medskip

\noindent
We define three new matrix functions in

\medskip\medskip

\begin{definition} \label{def:jee}
We let $j(k)$, $e(x,k)$ and $\tilde{e}(x,k)$ be given by
\begin{align}
  \bM_+(k) - \widetilde{Y}(-\ii k) &= -j(k) Y_1 ,
  \quad k \in \C_+ ,
\label{(3.10)}\\[0.25cm]
  \tilde{e}(x,k) &= \frac{e^{\ii k x}}{2 \ii k} E(-\ii k)
       - \frac{e^{-\ii k x}}{2 \ii k} E(\ii k)
  = \frac{1}{k} \ma \sin k x & 0 \\
      -2 \varpi k \cos k x & \sin k x \am ,\quad k \in \C ,
\label{(3.16)}\\
  e(x,k) &= \tilde{e}'(x,k) = \frac12 e^{\ii k x} E(-\ii k)
                        + \frac12 e^{-\ii k x} E(\ii k)
  = \ma \cos k x & 0 \\
       2 \varpi k \sin k x & \cos k x \am ,\quad k \in \C .
\label{eq:exkdef}
\end{align}
\end{definition}

\medskip\medskip

\noindent
Through the definition of $\bM_-$, we obtain the extension of $j(k)$
from $\C_+$ to $\C_-$. It follows that
\begin{equation} \label{j-extention}
   j(k) = \left\{\begin{array}{ll} j(k) ,& k \in \C_+ \\
                  j(-k) ,& k \in \C_- ; \end{array}\right.
\end{equation}
that is, for $\Im k \neq 0$, $j(k)$ is even function $j(-k) = j(k)$
with discontinuity on the real line for $\Im k = 0$ with jump
\[
   j(k) - j(-k) ,\quad \Im k = 0 .
\]
Using (\ref{eq:Y1tYD}), we find that
\begin{equation} \label{eq:jEM}
  j(k) = -\frac{1}{\ii k} E(-\ii k) - \bM_+(k) Y_1^{-1} .
\end{equation}
Consequently, $j(k) = {\mathcal O}\left(\frac{1}{|k|^2}\right)$ and
$j(k) - j(-k) = {\mathcal O}\left(\frac{1}{|k|^3}\right)$, $k \in
\C_+$. The Weyl matrix directly determines $j$; this is, because the
two leading terms in the asymptotic expansion of $\bM$ determine $E$
and $Y_1$ (cf.~(\ref{eq:tilY})-(\ref{eq:Edef})).

%
%
%

\medskip\medskip

\noindent
As $T_0 Y_0 = Y_0$ we note that
\begin{equation}
   \boldsymbol{\phi}_0(x,k) -T_0 e(x,k) = \mathcal{O}(1)
\end{equation}
(cf.~(\ref{eq:phi0def})).

\medskip\medskip

\begin{proposition} \label{prop:Psiasymp}
The function $\Psi(x,k)$ has asymptotic expansion
\begin{equation} \label{(3.11)}
  \Psi(x,k) = T_0 + \frac{1}{\ii k} \left\{
  \left(\bD(x) + \frac{\omega^2}{2 \mu_0} x\right) T_0
  -  T_1\right\}
            + o\left(\frac{1}{|k|}\right) ,\quad
  \Im k \leq 0 ,
\end{equation}
where $\bD(x)$ is given in Lemma~\ref{L-FF-1} and $T_0,T_1$ are defined by (\ref{eq:MT01-1}). It
admits the representation
\begin{multline}
   \Psi(x,k) =T_0 - \frac{1}{\pi} \int_{-\infty}^{\infty}
   \frac{k' e^{\ii k' x} \widehat{\boldsymbol{\varphi}}(x,k')
       (j(k') - j(-k')) E(-\ii k')}{k' - k} \dd k'
\\
   + \sum_{j=1}^N \left(\frac{e^{\ii k_j x}
     \boldsymbol{\varphi}(x,k_j) C_j E(-\ii k_j)}{
     2 \ii (k - k_j)}
   + \frac{e^{-\ii k_j x} \boldsymbol{\varphi}(x,k_j)
     C_j E(\ii k_j)}{2 \ii (k + k_j)}\right) .
\label{(3.12)}
\end{multline}
Moreover, the boundary values $\Psi_\pm(x,k) = \Psi(x,k \pm \ii 0)$,
$\Im k = 0$, determine $\boldsymbol{\varphi}$ in (\ref{eq:WvphiS}) by
\begin{equation} \label{(3.13)}
  2 \boldsymbol{\varphi}(x,k)
  = e^{-\ii k x} \Psi_+(x,k) E(-\ii k)^{-1}
                 + e^{\ii k x} \Psi_-(x,-k) E(\ii k)^{-1} .
\end{equation}
\end{proposition}

\begin{proof}
Substituting (\ref{eq:MT01-1})-(\ref {eq:MT01-2}) into
Corollary~\ref{cor:Phipmasym} yields (\ref{(3.11)}) for $\Im k < 0$
and by taking a one-sided limit, we get (\ref{(3.11)}) for $\Im k =
0$. For real-valued $k$, $\Im k = 0$, an explicit calculation gives
\begin{equation} \label{eq:Psi+-}
   \Psi(x,k + \ii \, 0) - \Psi(x,k - \ii \, 0) = -\ii k e^{\ii k x}
   \boldsymbol{\varphi}(x,k) (j(k) - j(-k)) E(-\ii k) .
\end{equation}
Hence,
\begin{multline} \label{eq:Psipmhvphi}
   \Psi(x,k + \ii \, 0)
   + \ii k  e^{\ii k x} \boldsymbol{\phi}_0(x,k) (j(k) - j(-k)) E(-\ii k)
   - \Psi(x,k - \ii \, 0)
\\
   = -\ii k e^{\ii k x}
     \widehat{\boldsymbol{\varphi}}(x,k) (j(k) - j(-k)) E(-\ii k)
   = \mathcal{O}\left(\frac{1}{|k|}\right)
\end{multline}
cf.~(\ref{eq:hatvphi}). This implies that the leading order in
asymptotics (\ref{(3.11)}) is the same for $\Psi(x,k - \ii \, 0)$ and
$\Psi(x,k + \ii \, 0) + \ii k \boldsymbol{\phi}_0(x,k) (j(k) - j(-k))
E(-\ii k)$ in the full complex $k$-plane and also (as $\Psi$ is
bounded on $\C$) that $\Psi$ can be recovered from the Cauchy integral
formula and its residues.


Now, we prove representation (\ref{(3.12)}). We let
%
%
$R > 0$ and denote by $\gamma^\pm_R$ the closed half-circle contour in
$\C_\pm$ with positive orientation containing all the poles $\pm k_j$
inside. We denote by $\Gamma^\pm_R$ only the arc parts of these
contours with negative orientation.
%
We have 
\begin{multline*}
   \frac{1}{2 \pi \ii} \int_{-R}^R
   \frac{\Psi_+(x,k') + \ii k' e^{ik'x}
     \boldsymbol{\phi}_0(x,k') (j(k') - j(-k')) E(-\ii k')
                              - \Psi_-(x,k')}{k - k'} \dd k'
\\
   = \operatorname{Er}_R(x,k) - \Psi(x,k) + \bD_0(x) T_0
\\[0.15cm]
     + \sum_{j=1}^N \frac{\Res_{k' = k_j} (\Psi(x,k')
       + \ii k' e^{i k' x} \boldsymbol{\phi}_0(x,k')
             (j(k') - j(-k')) E(-\ii k'))}{k - k_j}
     + \sum_{j=1}^N \frac{\Res_{k' = -k_j} \Psi (x,k')}{k' + k_j} ,
\end{multline*} 
where
\begin{multline*}
   \operatorname{Er}_R(x,k)
   = \frac{1}{2 \pi \ii} \int_{\Gamma^+_R}
   \frac{\Psi_+(x,k') + \ii k' e^{\ii k'x}
     \boldsymbol{\phi}_0(x,k') (j(k') - j(-k')) E(-\ii k')
                   - T_0}{k - k'} \dd k'
\\
   + \frac{1}{2 \pi \ii} \int_{\Gamma^-_R}
            \frac{\Psi_-(x,k') - T_0}{k - k'} \dd k' .
\end{multline*}
Lemma~\ref{lem:Psipoles} provides us with the expressions for the
residues; we use that $j(k') - j(-k')$ is entire for $\Im k \ne 0$.

We substitute (\ref{eq:Psipmhvphi}) in the integrand of the
integral in the left-hand side and take the limit $R \rightarrow
\infty$. As
\[
   \lim_{R \rightarrow \infty} \operatorname{Er}_R(x,k) = 0 ,
\]
we obtain (\ref{(3.12)}).

Finally, using (\ref{eq:Y1tYD}) and that $\boldsymbol{\Phi}_-(x,-k) =
\boldsymbol{\Phi}_+(x,k)$ for real-valued $k$ by definition, we obtain
(\ref{(3.13)}).
\end{proof}

\medskip\medskip

\noindent
We factorize $\boldsymbol{\varphi}$ and
$\widehat{\boldsymbol{\varphi}}$, and introduce $\cA$ and
$\widehat{\cA}$ according to
\begin{equation} \label{(3.15)}
   \boldsymbol{\varphi}(x,k) = T_0 \cA(x,k) ,\quad
   \widehat{\boldsymbol{\varphi}}(x,k) = T_0 \widehat{\cA}(x,k) ,\
   \quad k \in \R
\end{equation}
and write
\[
  \cA_j(x) = \cA(x,k_j) ,\quad j = 1,\ldots,N .
\]
We note that
\begin{equation} \label{eq:AhAres}
   \Res_{\pm k_j \in \C_{\pm}} \widehat{\cA}(x,k) j(k)
                       = \Res_{\pm k_j \in \C_{\pm}} \cA(x,k) j(k) .
\end{equation}
We introduce
\begin{equation}
   \stackrel{\circ}{e}(x,k) = T_0^{-1} \left\{T_0 e(x,k)
     - \boldsymbol{\phi}_0(x,k) \left(1 +
                \frac{1}{2} \ii k (j(k) - j(-k))\right)\right\}
\end{equation} 
(cf.~(\ref{eq:phi0def})) in

\medskip\medskip

\begin{lemma} \label{lem:cAeqs}
The functions $\cA(x,k)$, $\widehat{\cA}(x,k)$ and $\cA_j(x)$, $j =
1,\ldots,N$, satisfy
\begin{multline}
  4 \widehat{\cA}(x,k) = 4 \stackrel{\circ}{e}(x,k)
  + \frac{1}{\pi \ii} \int_{-\infty}^\infty k'
    \widehat{\cA}(x,k') j(k')
    \left(\tilde{e}(x,k' - k) + \tilde{e}(x,k' + k)\right) \dd k'
\\
  -\sum_{j=1}^N \cA_j(x) C_j
    \left(\tilde{e}(x,k_j - k) + \tilde{e}(x,k_j + k)\right) ,\quad
  \Im k = 0 .
\label{(3.18)}
\end{multline}
This equation holds in particular for $k = k_i$, $i = 1,\ldots,N$.
\end{lemma}

\begin{proof}
We take (\ref{(3.12)}) as a point of departure. We first study the
behavior of this equality, upon multiplication by $\ii k \, T_0^{-1}$,
as $k \rightarrow \infty$, in particular, the integral contribution
(suppressing the factor $\frac{1}{2 \pi}$) in the right-hand side,
\[
   \int_{-\infty}^{\infty}
   \frac{k' e^{\ii k' x} \widehat{\boldsymbol{\varphi}}(x,k')
       (j(k') - j(-k')) E(-\ii k')}{k' - k} \dd k' .
\]
To establish uniform boundedness, it is sufficient to only consider
the ``leading'' order of $\widehat{\boldsymbol{\varphi}}(x,k')$, that
is, the second term in (\ref{exp_varphi}),
\begin{equation} \label{leading}
   k' e^{ik'x} (e^{\ii k' x} + e^{-\ii k'x}) \bm1(x)
       (j(k') - j(-k')) \, \ma 0 & 0 \\ -2 \varpi \ii k' & 0 \am
\end{equation}
or, in fact, the non-vanishing matrix element (suppressing the factor
$-2 \varpi \ii$),
\[
   (k')^2 e^{2 \ii k'x} \bm1(x)
         (j(k') - j(-k')) + (k')^2 \bm1(x) (j(k') - j(-k')) ,
\]
in the numerator of the integrand. Concerning the second term, we note
that
\begin{equation} \label{eq:intj0}
   0 = \int_{-\infty}^{\infty} (k')^2 \bm1(x) (j(k') - j(-k')) \dd k'
   = \lim_{R \rightarrow \infty}
   \lim_{k \rightarrow \infty} (- k) \int_{-R}^R
     \frac{(k')^2 \bm1(x) (j(k') - j(-k'))}{k' - k} \dd k' ,
\end{equation}
which holds true also for $x = 0$. Hence,
\begin{equation}
   \lim_{R \rightarrow \infty}
   \lim_{k \rightarrow \infty} (-\ii) k T_0^{-1} \int_{-R}^R
   \frac{k' \widehat{\boldsymbol{\varphi}}(0,k')
       (j(k') - j(-k')) E(-\ii k')}{k' - k} \dd k' = 0 .
\end{equation}

We now let $x > 0$ and perform integration by parts, exploiting the
exponent $e^{2 \ii k' x}$, while analyzing the first term,
\begin{multline*}
   \int_{-R}^R \frac{ (k')^2 \bm1(x) e^{2 \ii k' x}
        (j(k') - j(-k'))}{k' - k} \dd k'
\\
   = -\int_{-R}^R \frac{\bm1(x)}{2 \ii x} e^{2 \ii k' x}
      \frac{\frac{\dd\ \ }{\dd k'} \left\{ (k')^2
        (j(k') - j(-k')) \right\}}{k' - k} \dd k'
      + \operatorname{Er}_{R;1}(x,k)
\\
     + \int_{-R}^R \frac{\bm1(x)}{2 \ii x} e^{2 \ii k' x}
       \frac{(k')^2 (j(k') - j(-k'))}{(k' - k)^2} \dd k'
      + \operatorname{Er}_{R;2}(x,k) .
\end{multline*}
We have
\[
   \lim_{R \rightarrow \infty}\lim_{k \rightarrow \infty}
                   (-\ii) k \operatorname{Er}_{R;1,2}(x,k) = 0 ,
\]
whence
\begin{multline}
   \lim_{R \rightarrow \infty} \lim_{k \rightarrow \infty} (-\ii) k
   \int_{-R}^R \frac{ (k')^2 \bm1(x) e^{2 \ii k' x}
     (j(k') - j(-k'))}{k' - k} \dd k'
\\
   = \lim_{R \rightarrow \infty} (-\ii) \int_{-R}^R
                   \frac{\bm1(x)}{2 \ii x} e^{2 \ii k' x}
     \frac{\dd\ \ }{\dd k'} \left\{ (k')^2
        (j(k') - j(-k')) \right\} \dd k'
\quad\text{is bounded.}
\end{multline}
This implies that
\begin{multline} \label{eq:intvpbound}
  \lim_{R \rightarrow \infty} \lim_{k \rightarrow \infty}
   (-\ii) k T_0^{-1} \int_{-R}^R
   \frac{k' e^{\ii k' x} \widehat{\boldsymbol{\varphi}}(x,k')
       (j(k') - j(-k')) E(-\ii k')}{k' - k} \dd k'
\\
   = \lim_{R \rightarrow \infty} \int_{-R}^R \ii  k'
     e^{i k'x} T_0^{-1} \widehat{\boldsymbol{\varphi}}(x,k')
                (j(k') - j(-k')) E(-\ii k') \dd k'
\\
   = 2 \int_{-\infty}^{\infty} \ii k'
     \widehat{\cA}(x,k') j(k') e(x,k') \dd k'
   \quad\text{is bounded.}
\end{multline}
Here, we used that the integrand is even in $k'$.

We use (\ref{(3.13)}) and that $\boldsymbol{\varphi}$ is even in $k$
to get
\[
  4 \boldsymbol{\varphi}(x,k)
    = e^{-\ii k x} (\Psi_+(x,k) + \Psi_-(x,k)) E(-\ii k)^{-1}
      + e^{\ii k x} (\Psi_-(x,-k) + \Psi_+(x,-k)) E(\ii k)^{-1} .
\]
Adding and subtracting terms, we rewrite this equality as
\begin{equation} \label{eq:proof-lemma-9.9-1_0}
  4 \boldsymbol{\varphi}(x,k)
  = 4 T_0 e(x,k)
              - 2 \ii k \boldsymbol{\phi}_0(x,k) (j(k) - j(-k))
  + \mathbf{I}_1(x,k) + \mathbf{I}_2(x,k)
    + \mathbf{I}_3(x,k) + \mathbf{I}_4(x,k)
\end{equation}
or
\begin{equation} \label{eq:proof-lemma-9.9-1}
  4 \widehat{\boldsymbol{\varphi}}(x,k)
  = 4 T_0 \stackrel{\circ}{e}(x,k)
  + \mathbf{I}_1(x,k) + \mathbf{I}_2(x,k)
    + \mathbf{I}_3(x,k) + \mathbf{I}_4(x,k) ,
\end{equation}
where
\begin{eqnarray*}
  \mathbf{I}_1 &=& e^{-\ii k x} (\Psi_+(x,k) + \ii k e^{\ii k x}
  \boldsymbol{\phi}_0(x,k) (j(k) - j(-k)) E(-\ii k) -  T_0)
                    E^{-1}(-\ii k) ,
\\
   \mathbf{I}_2 &=& e^{-\ii k x} (\Psi_-(x,k) -T_0)
                    E^{-1}(-\ii k) ,
\\
   \mathbf{I}_3 &=& e^{\ii k x} (\Psi_+(x,-k) + \ii k e^{-\ii k x}
   \boldsymbol{\phi}_0(x,k) (j(k) - j(-k)) E(\ii k) - T_0)
                    E^{-1}(\ii k) ,
\\
   \mathbf{I}_4 &=& e^{\ii k x} (\Psi_-(x,-k) - T_0)
                    E^{-1}(\ii k) .
\end{eqnarray*}
We note that $\mathbf{I}_1$ and $\mathbf{I}_4$ have poles, $k_j$, in
the complex $k$ half plane $\C_+$ ($\Im k > 0$) and that
$\mathbf{I}_2$ and $\mathbf{I}_3$ have poles, $-k_j$, in the complex
$k$ half plane $\C_-$ ($\Im k < 0$). We introduce a positive
half-circle (with radius $R$) contour $\gamma^+_R$ in $\C_+$
containing poles $k_j$ and $k$ inside, and a negative half-circle
(with radius $R$) contour $\gamma^-_R$ in $\C_-$ containing poles
$-k_j$ and $k$ inside. Then
\begin{multline}
   \frac{1}{2\pi \ii} \lim_{R \to \infty} \int_{\gamma^+_R}
   \frac{e^{-\ii k x} (\Psi_+(x,k') + \ii k' e^{\ii k' x}
     \boldsymbol{\phi}_0(x,k')
                (j(k') - j(-k')) E(-\ii k') -  T_0)
         E(-\ii k)^{-1}}{k' - k} \dd k'
\\
   = \mathbf{I}_1(x,k)
   + \sum_{j=1}^N \frac{e^{-\ii k x} \Res_{k' = k_j} (\Psi_+(x,k')
     + \ii k' e^{\ii k' x}
       \boldsymbol{\phi}_0(x,k') (j(k') - j(-k')) E(-\ii k'))
                 E(-\ii k)^{-1}}{k_j - k} ,
\end{multline}
\begin{multline}
   \frac{1}{2\pi \ii} \lim_{R \to \infty} \int_{\gamma^-_R}
       \frac{e^{-\ii k x} (\Psi_-(x,k') -  T_0)
         E(-\ii k)^{-1}}{k' - k} \dd k'
\\
   = \mathbf{I}_2(x,k)
   + \sum_{j=1}^N \frac{e^{-\ii k x} \Res_{k' = -k_j} \Psi_-(x,k')
         E(-\ii k)^{-1}}{-k_j - k} ,
\end{multline}
\begin{multline}
   \frac{1}{2\pi \ii} \lim_{R \to \infty} \int_{\gamma^-_R}
   \frac{e^{\ii k x} (\Psi_+(x,-k') + \ii k' e^{-\ii k' x}
     \boldsymbol{\phi}_0(x,k') (j(k') - j(-k')) E(\ii k')
         -T_0) E(\ii k)^{-1}}{k' - k} \dd k'
\\
   = \mathbf{I}_3(x,k)
   + \sum_{j=1}^N \frac{e^{\ii k x} \Res_{k' = -k_j} (\Psi_+(x,-k')
     + \ii k' e^{-\ii k' x}
       \boldsymbol{\phi}_0(x,k') (j(k') - j(-k')) E(\ii k'))
         E(\ii k)^{-1}}{-k_j - k} ,
\end{multline}
\begin{multline}
   \frac{1}{2\pi \ii} \lim_{R \to \infty} \int_{\gamma^+_R}
       \frac{e^{\ii k x} (\Psi_-(x,-k') - T_0)
         E(\ii k)^{-1}}{k' - k} \dd k'
\\
   = \mathbf{I}_4(x,k)
   + \sum_{j=1}^N \frac{e^{\ii k x} \Res_{k' = k_j} \Psi_-(x,-k')
         E(\ii k)^{-1}}{k_j - k} .
\end{multline}
We observe that the contributions from the upper and lower semicircles
in the limit vanish, that is,
\begin{multline}
   \frac{1}{2\pi \ii} \lim_{R \to \infty} \int_{\gamma^+_R}
   \frac{e^{-\ii k x} (\Psi_+(x,k') + \ii k' e^{\ii k' x}
     \boldsymbol{\phi}_0(x,k')
        (j(k') - j(-k')) E(-\ii k') - T_0)
         E(-\ii k)^{-1}}{k' - k} \dd k'
\\
   = \frac{1}{2\pi \ii} \int_{-\infty}^{\infty}
   \frac{e^{-\ii k x} (\Psi_+(x,k') + \ii k' e^{\ii k' x}
     \boldsymbol{\phi}_0(x,k')
        (j(k') - j(-k')) E(-\ii k') - T_0)
         E(-\ii k)^{-1}}{k' - k} \dd k' ,
\end{multline}
\begin{multline}
   \frac{1}{2\pi \ii} \lim_{R \to \infty} \int_{\gamma^-_R}
       \frac{e^{-\ii k x} (\Psi_-(x,k') -T_0)
         E(-\ii k)^{-1}}{k' - k} \dd k'
\\
   = -\frac{1}{2\pi \ii} \int_{-\infty}^{\infty}
       \frac{e^{-\ii k x} (\Psi_-(x,k') -T_0)
         E(-\ii k)^{-1}}{k' - k} \dd k' ,
\end{multline}
\begin{multline}
   \frac{1}{2\pi \ii} \lim_{R \to \infty} \int_{\gamma^-_R}
   \frac{e^{\ii k x} (\Psi_+(x,-k') + \ii k' e^{-\ii k' x}
     \boldsymbol{\phi}_0(x,k')
        (j(k') - j(-k')) E(\ii k') - T_0)
         E(\ii k)^{-1}}{k' - k} \dd k'
\\
   = -\frac{1}{2\pi \ii} \int_{-\infty}^{\infty}
   \frac{e^{\ii k x} (\Psi_+(x,-k') + \ii k' e^{-\ii k' x}
     \boldsymbol{\phi}_0(x,k')
        (j(k') - j(-k')) E(\ii k') -  T_0)
         E(\ii k)^{-1}}{k' - k} \dd k' ,
\end{multline}
\begin{multline}
   \frac{1}{2\pi \ii} \lim_{R \to \infty} \int_{\gamma^+_R}
       \frac{e^{\ii k x} (\Psi_-(x,-k') - T_0)
         E(\ii k)^{-1}}{k' - k} \dd k'
\\
   = \frac{1}{2\pi \ii} \int_{-\infty}^{\infty}
       \frac{e^{\ii k x} (\Psi_-(x,-k') - T_0)
         E(\ii k)^{-1}}{k' - k} \dd k' .
\end{multline}
Hence, using that in the residues, the terms containing
$\boldsymbol{\phi}_0(x,k') (j(k') - j(-k'))$ don't contribute,
\begin{multline} \label{eq:I1234}
   \mathbf{I}_1(x,k) + \mathbf{I}_2(x,k) + \mathbf{I}_3(x,k)
   + \mathbf{I}_4(x,k)
\\[0.25cm]
   = \frac{1}{2\pi \ii} \int_{-\infty}^{\infty}
   \frac{e^{-\ii k x} (\Psi_+(x,k') + \ii k' e^{\ii k' x}
     \boldsymbol{\phi}_0(x,k')
        (j(k') - j(-k')) E(-\ii k') - \Psi_-(x,k'))
         E(-\ii k)^{-1}}{k' - k} \dd k'
\\
   - \frac{1}{2\pi \ii} \int_{-\infty}^{\infty}
   \frac{e^{\ii k x} (\Psi_+(x,-k') + \ii k' e^{-\ii k' x}
     \boldsymbol{\phi}_0(x,k')
        (j(k') - j(-k')) E(\ii k') - \Psi_-(x,-k'))
         E(\ii k)^{-1}}{k' - k} \dd k'
\\
   - \sum_{j=1}^N \frac{e^{-\ii k x} \Res_{k' = k_j} \Psi_+(x,k')
     E(-\ii k)^{-1}}{k_j - k}
   + \sum_{j=1}^N \frac{e^{-\ii k x} \Res_{k' = -k_j} \Psi_-(x,k')
     E(-\ii k)^{-1}}{k_j + k}
\\
   + \sum_{j=1}^N \frac{e^{\ii k x} \Res_{k' = -k_j} \Psi_+(x,-k')
     E(\ii k)^{-1}}{k_j + k}
   - \sum_{j=1}^N \frac{e^{\ii k x} \Res_{k' = k_j} \Psi_-(x,-k')
     E(-\ii k)^{-1}}{k_j - k} .
\end{multline}
Using Lemma~\ref{lem:Psipoles}, the summations over the poles in
(\ref{eq:I1234}) add up to
\begin{multline} \label{eq:I1234-sums}
   - \sum_{j=1}^N \frac{e^{-\ii k x} \Res_{k' = k_j} \Psi_+(x,k')
     E^{-1}(-\ii k)}{k_j - k}
   + \sum_{j=1}^N \frac{e^{-\ii k x} \Res_{k' = -k_j} \Psi_+(x,k')
     E^{-1}(-\ii k)}{k_j + k}
\\
   + \sum_{j=1}^N \frac{e^{\ii k x} \Res_{k' = -k_j} \Psi_+(x,-k')
     E^{-1}(\ii k)}{k_j + k}
   - \sum_{j=1}^N \frac{e^{\ii k x} \Res_{k' = k_j} \Psi_-(x,-k')
     E^{-1}(-\ii k)}{k_j - k} .
\\
   = -\sum_{j=1}^N \boldsymbol{\varphi}(x,k_j) C_j
     (\tilde{e}(x,k_j - k) + \tilde{e}(x,k_j + k)) .
\end{multline}
Using (\ref{eq:Psi+-}), the integrals in (\ref{eq:I1234}) add up to
\begin{multline} \label{eq:I1234-ints}
   \frac{1}{2\pi \ii} \int_{-\infty}^{\infty}
   \frac{e^{-\ii k x} (\Psi_+(x,k') + \ii k' e^{\ii k' x}
     \boldsymbol{\phi}_0(x,k')
        (j(k') - j(-k')) E(-\ii k') - \Psi_-(x,k'))
         E(-\ii k)^{-1}}{k' - k} \dd k'
\\
   - \frac{1}{2\pi \ii} \int_{-\infty}^{\infty}
   \frac{e^{\ii k x} (\Psi_+(x,-k') + \ii k' e^{-\ii k' x}
     \boldsymbol{\phi}_0(x,k')
        (j(k') - j(-k')) E(\ii k') - \Psi_-(x,-k'))
         E(\ii k)^{-1}}{k' - k} \dd k'
\\
   = \frac{1}{\pi} \int_{-\infty}^{\infty} (-\ii k')
     \widehat{\boldsymbol{\varphi}}(x,k') j(k')
     (\tilde{e}(x,k' - k) + \tilde{e}(x,k' + k)) \dd k' . 
\end{multline}
We established boundedness of this integral in (\ref{eq:intvpbound}).
Substituting (\ref{eq:I1234-sums}) and (\ref{eq:I1234-ints}) into
(\ref{eq:I1234}) and the result into (\ref{eq:proof-lemma-9.9-1})
implies the statements upon considering $k \in \R$ (and $k = k_j$).
\end{proof}

\medskip\medskip

\noindent
With Proposition~\ref{prop:Psiasymp} and the proof of the previous
lemma concerning the limit $k \rightarrow \infty$ ($\Im k = 0$), we
obtain

\medskip\medskip

\begin{lemma} \label{lem:reconstr}
The following holds true
\begin{multline} \label{3.17}
   T_0^{-1} \left\{\left(\bD(x)
   + \frac{\omega^2}{2 \mu_0} x\right) T_0
   - T_1\right\}
\\
   = -\frac{1}{\pi} \int_{-\infty}^{\infty} \ii k'
     \widehat{\cA}(x,k') j(k') e(x,k') \dd k'
     + \sum_{j=1}^N \cA_j(x) C_j e(x,k_j) ,
\end{multline}
where  $\bD(x)$ is given in Lemma \ref{L-FF-1}.
\end{lemma}

\medskip\medskip

\noindent
In the above, we note that $T_1$ depends on $\omega$ through
$\theta_1$. This lemma provides an identity for $\bD$. \textit{It
  is essential for this lemma that we analyzed the asymptotic
  expansions of $\bF$, $\bM$ and $\boldsymbol{\Phi}$.} The right-hand side
of (\ref{3.17}) motivates the introduction of
\begin{multline} \label{eq:Kxydef}
   K(x,y) = -\frac{1}{\pi} \int_{-\infty}^{\infty} \ii k'
     \widehat{\cA}(x,k') j(k') e(y,k') \dd k'
     + \sum_{j=1}^N \cA_j(x) C_j e(y,k_j)
\\
   = -\frac{1}{2 \pi \ii} \int_{-\infty}^{\infty}
     \widehat{\cA}(x,k') (j(k') - j(-k')) e^{\ii k' y}
     E(-\ii k') k' \dd k'
     + \sum_{j=1}^N \cA_j(x) C_j e(y,k_j) ,\quad y \in [-x,x]
\end{multline}
(cf.~(\ref{eq:exkdef})), the first term of which can be identified
with a Fourier transform, and plays a key role in the proof of the
next proposition. The right-hand of (\ref{3.17}) is $K(x,x)$.

\medskip\medskip

\begin{remark} \label{rem:Kbounded}
From the analysis leading to (\ref{eq:intvpbound}) it follows that
$K(x,x)$ is bounded. In fact, the continuous differentiability of
$K(x,x)$ is directly related to the continuous differentiability of
$V$ through $\bD(x)$.
\end{remark}

\medskip\medskip

\noindent
We note that $T_0$ can be obtained from the asymptotic expansion of
$\bM$ (cf.~(\ref{(2.4)}) and (\ref{eq:MT01-1})). Suppose that $K(x,x)$
were known, then the potential, $V$, can be recovered upon
differentiating (\ref{3.17}):
\[
   T_0^{-1} \bD'(x) T_0^{-1} + \frac{\omega^2}{2 \mu_0} = K'(x,x) ,
\]
where $\bD$ is given in Lemma~\ref{L-FF-1} with
\begin{equation*}
   \bD'(x) = \frac12 V(x) \ma
   \displaystyle{-G_{11}^H \left(\frac{c_0}{2} G_{12}^H H
     + G_{22}^H \right)} &
   \displaystyle{ G_{11}^H \left(\frac{c_0}{2} G_{11}^H H
     + G_{21}^H \right)}
   \\[0.15cm]
   \displaystyle{-G_{12}^H \left(\frac{c_0}{2} G_{12}^H H
     + G_{22}^H \right)} &
   \displaystyle{ G_{12}^H \left(\frac{c_0}{2} G_{11}^H H
     + G_{21}^H \right)} \am
\end{equation*}

The kernel, $K(x,y)$, is determined by the boundary spectral data,
which is the content of

\medskip\medskip

\begin{proposition}[Gel'fand-Levitan] \label{prop:Gelfand-Levitan}
The kernel $K(x,y)$ is the unique solution of the Gel'fand-Levitan
type equation
\begin{equation} \label{eq:GLteq}
   4 K(x,y) + 4 \widehat{g}(x,y)
   - \int_{-x}^x K(x,y') E(2 \delta(x + y')) g(-y',y) \dd y' = 0 ,
   \quad y \in [-x,x] ,
\end{equation}
where $E$ is given in (\ref{eq:Edef}),
\begin{equation}
   g(x,y) \, = \, \frac{1}{\pi \ii}
           \int_{-\infty}^\infty e(x,k) j(k) e(y,k) k \dd k
             - \sum_{j=1}^N e(x,k_j) C_j e(y,k_j) .
\end{equation}
and
\begin{equation}
   \widehat{g}(x,y) = \frac{1}{\pi \ii}
   \int_{-\infty}^\infty \stackrel{\circ}{e}(x,k) j(k) e(y,k) k \dd k
             - \sum_{j=1}^N \stackrel{\ast}{e}(x,k_j) C_j e(y,k_j) ,
\end{equation}
in which
\begin{equation}
   \stackrel{\ast}{e}(x,k) =  e(x,k)
   - \frac{1}{2} \ii k T_0^{-1} \boldsymbol{\phi}_0(x,k)
             (j(k) - j(-k)) .
\end{equation}
\end{proposition}

\begin{proof}
We distinguish two parts to complete the proof. \textbf{Part I:
  construction of (\ref{eq:GLteq}).} We consider (\ref{eq:Kxydef}) and
write
\[
  K(x,y) = T_+(x,y) + T_-(x,-y) ,
\]
where
\[
   T_\pm(x,y) = -\frac{1}{2 \pi \ii} \int_{-\infty}^\infty
   \widehat{\cA}(x,k) j(k) e^{\ii k y} E(\mp \ii k) k \dd k
   + \frac{1}{2} \sum_{j=1}^N
             \cA_j(x) C_j e^{\ii k_j y} E(\mp \ii k_j) .
\]
We note that (cf.~(\ref{eq:AhAres}))
\begin{multline} \label{T_-}
   T_\pm(x,y) = -\frac{1}{2 \pi \ii} \int_{-\infty}^\infty
       \widehat{A}(x,k) j(k) e^{\ii k y} E(\mp\ii k) k \dd k
       + \sum_j \Res_{k_j \in \C_+} \widehat{\cA}(x,k)
       j(k) e^{\ii k y} E(\mp\ii k) k
\\
   = \lim_{R \rightarrow \infty} \frac{1}{2 \pi \ii}\int_{\Gamma^+_R}
     \widehat{\cA}(x,k) e^{\ii k x} j(k) e^{\ii k (y - x)}
     E(\mp\ii k) k \dd k
\end{multline}
and
\begin{multline} \label{T_-bis}
   T_\pm(x,y) = -\frac{1}{2 \pi \ii} \int_{-\infty}^\infty
       \widehat{\cA}(x,k) j(k) e^{\ii k y} E(\mp\ii k) k \dd k
       - \sum_j \Res_{-k_j \in \C_-} \widehat{\cA}(x,k)
       j(k) e^{\ii k y} E(\mp\ii k) k
\\
   = -\lim_{R \rightarrow \infty} \frac{1}{2 \pi \ii} \int_{\Gamma^-_R}
      \widehat{\cA}(x,k) e^{-\ii k x} j(k) e^{\ii k (y + x)}
      E(\mp\ii k) k \dd k .
\end{multline}
with
\begin{equation}
   - \sum_j \Res_{-k_j \in \C_-} \widehat{\cA}(x,k)
            j(k) e^{\ii k y} E(\mp\ii k) k
   = \sum_j \Res_{k_j \in \C_+} \widehat{\cA}(x,k)
       j(k) e^{\ii k y} E(\mp\ii k) k
\end{equation}
as
\[
   \lim_{k \rightarrow k_j} (k - k_j) j(k)
                  = -\lim_{k \rightarrow -k_j} (k + k_j) j(k) .
\]
The absence of singularities means that $\cA(x,k) e^{\pm \ii k x}$ has
a bounded holomorphic extension to the half-plane $\pm \Im k \ge
0$. Through exponential decay of $e^{\ii k (y - x)}$, we find that
\[
  \text{if $|y| > x$ then $T_\pm(x,y) = 0$ .}
\]

From the Fourier inversion formula, we obtain
\[
  \widehat{\cA}(x,k) j(k) k
          - \pi \ii \sum_{j=1}^N \cA_j(x) C_j \delta(k - k_j)
  = -\left(\ii \int_{-x}^x T_\pm(x,y) e^{-\ii k y} \dd y
                  \right) E(\pm \ii k) .
\]
We substitute this expression in (\ref{(3.18)}). Using that
\begin{multline*}
   \frac{1}{\pi \ii} \int_{-\infty}^\infty k'
   \left[\mstrut{0.6cm}\right.
     - \pi \ii \sum_{j=1}^N \cA_j(x) C_j \delta(k - k_j)
   \left.\mstrut{0.6cm}\right]
   j(k') (\tilde{e}(x,k' - k) + \tilde{e}(x,k' + k)) \dd k'
\\
   = \sum_{j=1}^N \cA_j(x) C_j
        (\tilde{e}(x,k_j - k) + \tilde{e}(x,k_j + k)) ,
\end{multline*}
which shows that the summation over poles in (\ref{(3.18)}) is
cancelled, we obtain
\begin{equation} \label{eq:cApmreprs}
   4 \widehat{\cA}(x,k) = 4 \stackrel{\circ}{e}(x,k)
                - \int_{-x}^x T_\pm (x,y) B_\pm(k,y) \dd y ,
\end{equation}
in which
\begin{multline*}
   B_\pm(k,y) = \frac{1}{2 \pi \ii} \int_{-\infty}^\infty
   e^{-\ii k' y} E(\pm \ii k')
   \left(\frac{e^{\ii (k' - k) x} E(-\ii (k' - k))}{k' - k}
      - \frac{e^{-\ii (k' - k) x} E(\ii (k'- k))}{k' - k} \right.
\\ \left.
      + \frac{e^{\ii (k' + k) x} E(-\ii (k' + k))}{k' + k}
      - \frac{e^{-\ii (k' + k) x} E(\ii (k' + k))}{k' + k}\right)
   \dd k' .
\end{multline*} 
By straightforward calculations, we find that
\begin{multline} \label{eq:B+}
   B_+(k,y) = (\sgn(x - y) e(y,k) + \sgn(x + y) e(-y,k))
\\
   + 8 \varpi \delta(x + y) e(-y,k) \ma 0 & 0 \\ 1 & 0\am
   + 4 \varpi \sgn(x + y) \dede{}{y} e(-y,k)
       \ma 0 & 0 \\ 1 & 0\am
\end{multline}
and
\begin{multline} \label{eq:B-}
   B_-(k,y) = B_+(k,-y) = (\sgn(x + y) e(-y,k) + \sgn(x - y) e(y,k))
\\
   + 8 \varpi \delta(x - y) e(y,k) \ma 0 & 0 \\ 1 & 0\am
   + 4 \varpi \sgn(x - y) \dede{}{y} e(y,k)
       \ma 0 & 0 \\ 1 & 0\am .
\end{multline}
Taking half the sum of the $\pm$ representations in
(\ref{eq:cApmreprs}), we get
\[
  4 \widehat{\cA}(x,k) - 4 \stackrel{\circ}{e}(x,k)
             = -\frac12 \int_{-x}^x K(x,y) B_+(k,y) \dd y .
\]
Substituting (\ref{eq:B+}) and (\ref{eq:B-}) into this equation and
using that $e(x,k) + e(-x,k) = 2 \cos k x \, I_2$, we get 
\begin{equation} \label{3.27}
   4 \widehat{\cA}(x,k) - 4 \stackrel{\circ}{e}(x,k)
     = -\int_{-x}^x K(x,y') E(2 \delta(x + y')) e(-y',k) \dd y' .
\end{equation}
We multiply this equation by $2 j(k) e(y,k) k - 2 \pi \ii \sum_{j=1}^N
\delta (k-k_j) C_j e(y,k_j)$ and obtain
\begin{multline*}
   4 \widehat{\cA}(x,k) j(k) 2 e(y,k) k
   - 8 \pi \ii \cA(x,k) \sum_{j=1}^N \delta(k - k_j) C_j e(y,k_j)
\\[0.25cm]
   - 4 \stackrel{\circ}{e}(x,k) j(k) 2 e(y,k) k
                        + 8 \pi \ii \stackrel{\ast}{e}(x,k)
              \sum_{j=}^N \delta (k - k_j) C_j e(y,k_j)
\\
   = -\int_{-x}^x K(x,y') E(2 \delta(x + y')) e(-y',k) \dd y'
   \left(\mstrut{0.5cm}\right.
   j(k) 2 e(y,k) k
   - 2 \pi \ii \sum_{j=1}^N \delta (k-k_j) C_j e(y,k_j)
   \left.\mstrut{0.5cm}\right) .
\end{multline*}
Dividing this equation by $2 \pi \ii$ and integrating over $k$, leads
to (\ref{eq:GLteq}). More precisely, we first integrate over $[-R,R]$
and establish that the integrals are uniformly bounded after which we
interchange orders of integration and take the limit $R \rightarrow
\infty$.

As
\[
   \int_{-x}^x \left(\begin{array}{cc} 0 & 0 \\
   2 \varpi k \sin(k y') & 0 \end{array}\right)
   \left(I_2 + \left(\begin{array}{cc}
     0 & 0 \\ 4 \varpi \delta(x + y') & 0 \end{array}\right)\right)
   \cos(k' y') \dd y' = 0
\]
the representation of $g$, in fact, can be simplified,
\begin{equation}
   g(x,y) \, = \, \frac{1}{\pi \ii}
           \int_{-\infty}^\infty \cos(k x) j(k) e(y,k) k \dd k
             - \sum_{j=1}^N \cos(k_j x) C_j e(y,k_j) .
\end{equation}
   
\medskip\medskip

\noindent
\textbf{Part II: (\ref{eq:GLteq}) has a unique solution.} We note that
$g(x,y)=0$ for $|y|>x$ and that equation (\ref{eq:GLteq}) is of
Volterra type. We consider $x$ as parameter and $K(x,y)$ as unknown
function. For unique solvability, we need to prove that, for some
constant $C > 0$ (dependent on $x$),
\begin{equation} \label{boundkernel}
   \sup_{|y| \le x} \int_{-x}^x |E(2 \delta(x + y')) g(-y',y)| \dd y'
   \le C .
\end{equation}
Using special form of matrices $g$ and $E,$ it follows that
(\ref{boundkernel}) is satisfied. Then, using the Volterra property,
it follows that the solution to the homogeneous problem is trivial and
the solution to (\ref{eq:GLteq}) can be constructed by iteration. This
completes the proof of Proposition~\ref{prop:V}.
\end{proof}  

\subsection{Recovery of $G^H$}
\label{ssec:7.2}

Here, we prove that $G^H$ is determined by the two leading orders in
asymptotic expansion of the Jost solution $\bF$ at $x=0$ as $\xi \to
\infty$, $\xi \in \cK_+$. The asymptotic expansion of $\bF(0,\xi)$ is
given by (\ref{xilimitJost})-(\ref{xilimitJost-G1}) upon substituting
$x = 0$:
\begin{align*}
   \bF(0,\xi) &= \xi
   \left(-\frac{\mu_0}{\omega^2} \ma G_{11}^H & G_{11}^H \\[0.25cm]
   G_{12}^H & G_{12}^H \am \right.
\\
   &\left. + \frac{1}{\xi} \left\{\ma
   \displaystyle{G_{11}^H \frac{(\lambda_0 + \mu_0) H}{
     2(\lambda_0 + 2 \mu_0)} + G_{21}^H} & 0 \\
   \displaystyle{G_{12}^H \frac{(\lambda_0 + \mu_0) H}{
     2 (\lambda_0 + 2 \mu_0)} + G_{22}^H} & 0 \am
   - \frac{1}{2} \frac{\mu_0}{\omega^2} \int_0^H V(y) \dd y
   \ma G_{11}^H & G_{11}^H \\[0.25cm] G_{12}^H & G_{12}^H \am
   \right\}\right.
\nonumber\\
   &\left. + o\left(\frac{1}{|\xi|}\right)\right) .
\nonumber
\end{align*}
The expression for $V(y)$ can be directly deduced from the
analysis in Section~\ref{sec:Mark}. From the leading order term in
this asymptotic expansion, we recover $G_{11}^H$ and $G_{12}^H$.

The next order term, we write as $S(\omega)=X + \omega^{-2} Y$, where
$X$ and $Y$ are independent of $\omega$ and given by
\begin{multline}
   X = \ma
   \displaystyle{G_{11}^H \frac{(\lambda_0 + \mu_0) H}{
     2 (\lambda_0 + 2 \mu_0)} + G_{21}^H} & 0 \\
   \displaystyle{G_{12}^H \frac{(\lambda_0 + \mu_0) H}{
     2 (\lambda_0 + 2 \mu_0)} + G_{22}^H} & 0 \am
   - \frac{\mu_0}{2} \int_0^H \left(\mstrut{0.55cm}\right.
   (G^{-1}(y) B_2(x) G(x))^{\rm T}
\\
   - (G_0^{-1}(y) \left(\begin{array}{cc}
     \displaystyle{-\frac{1}{\mu_0}} & 0 \\
     0 & \displaystyle{-\frac{1}{\lambda_0 + 2 \mu_0}}
                        \end{array}\right)
   G_0(x))^{\rm T} \left.\mstrut{0.55cm}\right) \dd y
   \ma G_{11}^H & G_{11}^H \\[0.25cm] G_{12}^H & G_{12}^H \am 
\end{multline}
and
\begin{align}
   Y = -\frac{1}{2} \mu_0 \int_0^H
        \left(G^{-1} B_1 G\right)^{\rm T} \dd y
   \ma G_{11}^H & G_{11}^H \\[0.25cm] G_{12}^H & G_{12}^H \am .
\end{align}
Using $S(\omega_1)$, $S(\omega_2)$ for any two frequencies
$\omega_1 \neq \omega_2 \in \R_+$, we get
\[
   Y = \frac{1}{\frac{1}{\omega_1^2} - \frac{1}{\omega_2^2}}
               \left(S(\omega_1) - S(\omega_2)\right)
\]
and then, simply, $X = S(\omega_1) - \omega_1^{-2} Y$. We multiply $X$
from the right with $\ma 1 & 1 \\ -1 & -1 \am$ and obtain
\begin{equation}
   X \ma 1 & 1 \\ -1 & -1 \am = \ma
   \displaystyle{G_{11}^H \frac{(\lambda_0 + \mu_0) H}{
     2 (\lambda_0 + 2 \mu_0)}} &
   \displaystyle{G_{11}^H \frac{(\lambda_0 + \mu_0) H}{
     2 (\lambda_0 + 2 \mu_0)}} \\
   \displaystyle{G_{12}^H \frac{(\lambda_0 + \mu_0) H}{
     2 (\lambda_0 + 2 \mu_0)}} &
   \displaystyle{G_{12}^H \frac{(\lambda_0 + \mu_0) H}{
     2 (\lambda_0 + 2 \mu_0)}} \am
   + \ma G_{21}^H & G_{21}^H \\[0.25cm] G_{22}^H & G_{22}^H \am .
\end{equation}
As we already recovered $G_{11}^H$ and $G_{12}^H$, and $G_{11}^H
G_{22}^H - G_{12}^H G_{21}^H = 1$, we obtain $G_{21}^H$ and
$G_{22}^H$.

\subsection{Recovery of $\lambda$ and $\mu$}
\label{ssec:7.3}

With the recovery of $G^H$ we recover $Q_0$ and, hence, $Q$. Finally,
we note that $Q = Q(\omega)= Q_1 + \omega^2 Q_2$ with $Q_j$ related
to $B_j^{\rm T}$ in (\ref{1rw(2.9a)new})-(\ref{1rw(2.9b)new}) by
similarity transformations. Then, if $Q(\omega_1),$ $Q(\omega_2)$ are
known for some frequencies $\omega_1 \neq \omega_2 \in \R_+$, we obtain
\[
   Q_2 = \frac{1}{\omega_1^2 - \omega_2^2}
                 \left(Q(\omega_1) - Q(\omega_2)\right)
\]
and then
\[
   \tr Q_2  = -\frac{1}{\mu} - \frac{1}{\lambda + 2 \mu}
   \quad\text{and}\quad
   \det Q_2 = \frac{1}{\mu} \, \frac{1}{\lambda + 2 \mu} ,
\]
wherefrom $\lambda$ and $\mu$ are recovered.

\section*{Acknowledgements}


MVdH was supported by the Simons Foundation under the MATH + X
program, the National Science Foundation under grant DMS-1815143, and
the corporate members of the Geo-Mathematical Imaging Group at Rice
University.







\section{Data Availability}

Data sharing is not applicable to this article as no new data were created or analyzed in this study

\appendix

\renewcommand{\theequation}{\thesection.\arabic{equation}}
\renewcommand\thesection{\Alph{section}}
\setcounter{equation}{0}

\section{Riemann surface}
\label{app:RS}

For the introduction of the proper Riemann surface, we refer to
Chapman \cite{Chapman1972}. We denote by $\sqrt{z}$ the principal
branch of the square root that is positive for $z>0$ and with the cut
along the negative real axis. For the analytic continuation in $|\xi|
\in \R_+$ we replace $|\xi|$ by $\xi \in \C$. We define $q_S(\xi)$ by
choosing the branch with
\[
   q_S(\xi) \in \ii \R_+ \,\, \text{for real-valued} \,\,
   \xi > \frac{\omega}{\sqrt{\mu_0}}\quad\text{and}\quad
   q_S(\xi) \in \ii \R_- \,\, \text{for real-valued} \,\,
   \xi < -\frac{\omega}{\sqrt{\mu_0}} .
\]
Then 
\[
   \Im q_S(\xi) > 0 \,\, \text{for} \,\,
       \Re \xi > \frac{\omega}{\sqrt{\mu_0}}\quad\text{and}\quad
   \Im q_S(\xi) < 0 \,\, \text{for} \,\,
       \Re \xi < -\frac{\omega}{\sqrt{\mu_0}} .
\]
We note that $\Im q_S(\xi) = 0$ for $\xi \in
[-\frac{\omega}{\sqrt{\mu_0}},\frac{\omega}{\sqrt{\mu_0}}] \cup \ii
\R$.

We let~\footnote{\textbf{[In Aki \& Richards \cite{AkiRichards} the
      branch cuts are chosen as $([0,\tfrac{\omega}{\sqrt{\mu_0}}]
      \cup \ii \R_+) \cup ([-\tfrac{\omega}{\sqrt{\mu_0}},0] \cup \ii
      \R_-)$]}}
\[
   \cK_S := \C \setminus
   \left(\left[-\frac{\omega}{\sqrt{\mu_0}},
         \frac{\omega}{\sqrt{\mu_0}}\right] \cup
   \ii \R \right) .
\]
We observe that $q_S :\ \xi \to \sqrt{\frac{\omega^2}{\mu_0} - \xi^2}$
is a conformal mapping $\cK_S \to \cK_S$ and satisfies
\begin{equation} \label{limitS}
  q_S(\xi) = \ii \xi - \frac{\ii \omega^2}{2 \mu_0 \xi}
  + {\mathcal O}\left(\frac{1}{|\xi|^2}\right)\quad
  \text{as} \quad |\xi| \to \infty .
\end{equation}
Moreover, $q_S$ maps the cut $[-\frac{\omega}{\sqrt{\mu_0}},
  \frac{\omega}{\sqrt{\mu_0}}]$ on the real axis onto itself, and the
imaginary axis onto the complement of this cut on the real axis:
\begin{equation} \label{comslitmap}
  q_S(\ii \R) = \left(-\infty,-\frac{\omega}{\sqrt{\mu_0}}\right]
      \cup \left[\frac{\omega}{\sqrt{\mu_0}},\infty\right) .
\end{equation}
Furthermore,
\[
   q_S(\ii \R_\pm) = \R_\mp \setminus
   \left(-\frac{\omega}{\sqrt{\mu_0}},
            \frac{\omega}{\sqrt{\mu_0}}\right) ,\quad
   q_S\left(\R_\pm \setminus \left(-\frac{\omega}{\sqrt{\mu_0}},
            \frac{\omega}{\sqrt{\mu_0}}\right) \right) = \ii \R_\pm .
\]   
and
\[
   \pm \Im(q_S(\xi)) > 0 \,\, \text{iff}  \,\,
   \xi \in \cK_{S,\pm} = \{\xi \in \cK_S\ :\ \pm \Re \xi > 0 \} .
\]

The Riemann surface for $q_S(\xi)$ is obtained by joining the upper
and lower rims of two copies of $\C \setminus
[(-\infty,-\frac{\omega}{\sqrt{\mu_0}}] \cup
[\frac{\omega}{\sqrt{\mu_0}},\infty)]$ cut along the
$(-\infty,-\frac{\omega}{\sqrt{\mu_0}}] \cup
      [\frac{\omega}{\sqrt{\mu_0}},\infty)$ in the usual (crosswise)
way. Instead of this two-sheeted Riemann surface it is more convenient
to work on the cut plane $\cK_S$ and half planes $\cK_{S,\pm}$ such
that $q_S(\cK_{S,\pm}) = \C_\pm := \{ z \in \C\ :\ \pm \Im z >
0\}$. The ``upper'' (physical) sheet for $q_S$ corresponds to
$\cK_{S,+}$.

We collect below some useful properties
\begin{equation} \label{k-properties10}
\begin{aligned}
  \quad \Im  q_S(\xi) > 0 \,\, \text{iff} \,\, \xi \in \cK_{S,+} ,
  \\[0.25cm]
  \text{for} \,\, \xi \in \C \setminus
  \left(\left[-\frac{\omega}{\sqrt{\mu_0}},
    \frac{\omega}{\sqrt{\mu_0}}\right] \cup \ii \R\right) : &
  \quad q_S(\xi) = -q_S(-\xi) = -\overline{q_S(\overline \xi)} ,
  \\
  \text{for} \,\, \xi \in \left[-\frac{\omega}{\sqrt{\mu_0}},
    \frac{\omega}{\sqrt{\mu_0}}\right] : &
  \quad q_S(\xi \pm \ii \, 0)
    = \mp\left|\frac{\omega^2}{\mu_0} - \xi^2\right|^{1/2} ,
  \\
  \text{for} \,\, \xi \in \ii\R , : & \quad q_S(\xi \pm 0)
    = \mp\left|\frac{\omega^2}{\mu_0} - \xi^2\right|^{1/2} ,
  \\
  \text{for} \,\, \xi \in
  \left(-\infty,-\frac{\omega}{\sqrt{\mu_0}}\right] \cup
        \left[\frac{\omega}{\sqrt{\mu_0}},\infty\right) : &
  \quad q_S(\xi) = \pm \ii \left|\xi^2
        - \frac{\omega^2}{\mu_0}\right|^{1/2} ,
  \quad \pm \xi \ge \frac{\omega}{\sqrt{\mu_0}} .
\end{aligned}
\end{equation}

\medskip\medskip

\noindent
By replacing $\mu_0$ with $\sigma_0 := \lambda_0 + 2 \mu_0$ we get
analogous properties for quasimomentum $$q_P(\xi) =
\sqrt{\frac{\omega^2}{\sigma_0} - \xi^2} .$$ Corresponding objects get
the subscript $P$ instead of $S$. We introduce
\[
   \cK_P := \C \setminus
   \left(\left[-\frac{\omega}{\sqrt{\sigma_0}},
     \frac{\omega}{\sqrt{\sigma_0}}\right] \cup \ii \R\right) .
\]
We observe that $q_P :\ \xi \to \sqrt{\frac{\omega^2}{\sigma_0} -
  \xi^2}$ is a conformal mapping $\cK_P \to \cK_P$. We obtain the
Riemann surface $\cR$ for both $q_P$ and $q_S$ by joining the separate
Riemann surfaces for $q_P$ and $q_S$ so that $q_P$ and $q_S$ are
single-valued holomorphic functions of $\xi$. We note that $\cR$ is a
four-fold cover of the plane. We identify the part of $\cR$ where $\Im
q_P > 0$, $\Im q_S > 0$ with the physical (``upper'') sheet which
coincides with $\cK_{S,+}$. Each sheet can be identified by the signs
of $\Im q_S$ and $\Im q_P$. We omit the subscript $S$ in the notation
and write $\cK = \cK_S$ and $\cK_+ = \cK_{S,+}$ for the cut plane and
the part of the cut plane corresponding to the physical sheet,
respectively. We note that $\xi$ has the meaning of Regge parameter.

\medskip\medskip

\noindent
In the main text, we introduce $\zeta = \xi^2$. We note that $\Im
q_S(\zeta) > 0$, $\Im q_P(\zeta) > 0$ for $\zeta \in \Pi_+$, where
\begin{equation} \label{eq:Pi+}
   \Pi_+ = \C \setminus \left(-\infty,\frac{\omega^2}{\mu_0}\right]
\end{equation}
corresponds to the physical sheet while $(\cK_+)^2 = \Pi_+$. We
introduce the notation
\[
   \Pi_{+,1} = \overline{\Pi_+} \setminus
           \left\{\frac{\omega^2}{\mu_0}\right\} ,\quad
   \cK_{+,1} = \overline{\cK_+} \setminus
           \left\{\frac{\omega}{\sqrt{\mu_0}}\right\} .
\]
We will use both parameters $\xi$ (Jost solutions and Jost function)
and $\zeta$ (Weyl solutions and Weyl matrix) and both cut planes
$\cK_+$ and $\Pi_+$, switching between them when it appears natural.

\setcounter{equation}{0}

\section{Green's function}
\label{app:GF}

We have
\[
  \bcG(x,y) = [F^{(1)}(x,y) \,\, F^{(2)}(x,y)] ,
\]
where $F^{(1)}(x,y)$, $F^{(2)}(x,y)$ are solutions of
(\ref{StLgen_0}) for $x > y$, with boundary values
\[
  F^{(1)}(y,y) = \ma 0 \\ 0 \am ,\quad
  (F^{(1)})'(y,y) = \ma 1 \\ 0 \am ,\quad
  F^{(2)}(y,y) = \ma 0 \\ 0 \am ,\quad
  (F^{(2)})'(y,y) = \ma 0 \\ 1 \am .
\]
By explicit construction, we obtain

\medskip\medskip

\begin{lemma} \label{lem:bcGABC}
The following holds true
\begin{equation} \label{ABC}
  \bcG(x,y) =
  \bcA(x) \frac{\sin((x - y) q_P)}{q_P}
     + \bcB(y) \frac{\sin((x - y) q_S)}{q_S}
     + \bcC \frac{\cos((x - y) q_S) - \cos((x - y) q_P)}{\omega^2} ,
\end{equation}
where
\begin{align*}
  \bcA(x) &= \left(\begin{array}{cc}
    G_{12}^H \left(\frac{c_0}{2} G_{11}^H (x - H) - G_{21}^H\right) &
    G_{11}^H \left(-\frac{c_0}{2} G_{11}^H (x - H) + G_{21}^H \right)
    \\ \\
    G_{12}^H \left(\frac{c_0}{2} G_{12}^H (x - H) - G_{22}^H\right) &
    G_{11}^H \left(-\frac{c_0}{2} G_{12}^H (x - H) + G_{22}^H\right)
    \end{array}\right) .
\\[0.25cm]
  \bcB(y) &= \left(\begin{array}{cc}
  G_{11}^H \left(\frac{c_0}{2} (-y + H) G_{12}^H + G_{22}^H\right) &
   -G_{11}^H \left(\frac{c_0}{2} (-y + H) G_{11}^H + G_{21}^H\right)
  \\ \\
  G_{12}^H \left(\frac{c_0}{2} (-y + H) G_{12}^H + G_{22}^H\right) &
   -G_{12}^H \left(\frac{c_0}{2} (-y + H) G_{11}^H + G_{21}^H\right)
    \end{array}\right) ,
\\[0.25cm]
  \bcC &= \left(\begin{array}{cc}
    \mu_0 G_{12}^H G_{11}^H & -\mu_0 (G_{11}^H)^2 
    \\ \\
    \mu_0 (G_{12}^H)^2 & -\mu_0 G_{12}^H G_{11}^H 
    \end{array}\right)
\end{align*}
(cf.~(\ref{1rw(3.8)new})).
\end{lemma}

\medskip\medskip

We note that $\bcA(x)$ and $\bcB(y)$ are first-order matrix-valued
polynomials in $x$ and $y$, respectively, while $\bcC$ is a constant
matrix.

\subsection*{Homogeneous half space}

In a homogeneous half space when $H = 0$, with $\mu = \mu_0$ and
$G_{12}^H = G_{21}^H = 0$, $G_{11}^H = G_{22}^H = 1$, (\ref{ABC})
reduces to
\begin{multline}
  \bcG(x,y) =
  \left(\begin{array}{c} \displaystyle
  \frac{c_0}{2} (-y) \frac{\sin((x - y) q_S)}{q_S}
  \\ \\ \\
  0 \end{array}\right.
\\
  \left.\begin{array}{c} \displaystyle
  \displaystyle
  -\frac{c_0}{2} \left[x \frac{\sin((x - y) q_P)}{q_P}
                     - y \frac{\sin((x - y) q_S)}{q_S}\right]
  + \mu_0 \left[\frac{\cos((x - y) q_P)
                        - \cos((x - y) q_S)}{\omega^2}\right]
  \\ \\
  \displaystyle \frac{\sin((x - y) q_P)}{q_P} \end{array}\right) .
\end{multline}

\setcounter{equation}{0}

\section{Weyl matrix and Neumann-to-Dirichlet map of the Rayleigh
         system}
\label{app:ND}

In this appendix, we study relationships between the orginal and
transformed problems, that is, the Jost and Weyl solutions and the
Jost function before the Markushevich transform. Consistent with the
notation in (\ref{MPtr}), we let
\begin{equation}
  \mathbf{w}(x,\xi) = \gM^{-1}(\bF)(x,\xi) ,
\end{equation}
where $\bF$ signifies the Jost solution (cf.~(\ref{eq:Jostsoldef})),
and write 
\begin{equation} 
  \mathbf{w} = [w_P \,\, w_S]\quad
  \text{and}\quad
\widetilde{w}^-= [\widetilde{w}^-_P \,\, \widetilde{w}^-_S]
\end{equation}
(cf.~(\ref{relationBF})) supplemented with boundary conditions
(\ref{1rw(1.3)})-(\ref{1rw(1.4)}),
\begin{equation} \label{eq:Bwf}
  \bB(\mathbf{w}) = \ma b_-(w_P) & b_-(w_S) \\
         a_-(w_P) & a_-(w_S) \am
  = \ma 0 & \ii \\ -1 & 0 \am
    \ma a_-(\widetilde{w}^-_P) & a_-(\widetilde{w}^-_S) \\
         b_-(\widetilde{w}^-_P) & b_-(\widetilde{w}^-_S) \am ;
\end{equation}
in the right-most equality, we reverted to the original notation
(cf.~(\ref{relationBF})). Setting $\chi = \bB(\mathbf{w})$,
(\ref{1rw(2.12)bis}) yields the expression for the Jost function,
\begin{equation} \label{eq:JfB}
  \bF_{\Theta}(\xi) = (D^{\rm a}(\xi))^{-1} \bB(\mathbf{w}) .
\end{equation}
In a likewise manner, we obtain for the adjoint problem,
\begin{equation} \label{eq:JafB}
  \bF^{\rm a}_{\Theta}(\xi) = (D(\xi))^{-1} \bB(\mathbf{w}) .
\end{equation}
Substituting (\ref{eq:Bwf}) into (\ref{eq:JfB}) then gives
\begin{equation} \label{eq:FThetBf}
  \bF_{\Theta}(\xi) = \frac{1}{2 \mu_0 \mu(0) \xi}
     \ma \mu(0) & 0 \\ \displaystyle
     2 \mu_0 \frac{\mu'(0)}{\mu(0)} & 2 \mu_0 \xi \ii \am
     \ma a_-(\widetilde{w}^-_P) & a_-(\widetilde{w}^-_S) \\
         b_-(\widetilde{w}^-_P) & b_-(\widetilde{w}^-_S) \am .
\end{equation}
Substituting (\ref{eq:Bwf}) into (\ref{eq:JafB}) gives
\begin{equation} \label{eq:FaThetBf}
  \bF^{\rm a}_{\Theta}(\xi) = \frac{1}{2 \mu_0 \mu(0) \xi}
     \ma -2 \mu_0 \xi & 0 \\ 0 & -\mu(0) \ii \am
     \ma a_-(\widetilde{w}^-_P) & a_-(\widetilde{w}^-_S) \\
         b_-(\widetilde{w}^-_P) & b_-(\widetilde{w}^-_S) \am .
\end{equation}

\medskip\medskip

\noindent
We subject the Weyl solution to $\gM^{-1}$ (cf.~(\ref{1rw(2.10)new})),
substitute (\ref{eq:WvphiS}) and introduce
\begin{equation} \label{eq:rtttpM}
  \mathbf{r}(x,\xi) = \gM^{-1}(\boldsymbol{\Phi})(x,\xi)
  = \boldsymbol{\theta}(x,\xi) + \boldsymbol{\psi}(x,\xi) \bM(\xi) ,
\end{equation}
with
\begin{equation}
   \boldsymbol{\theta}(x,\xi) = \gM^{-1}(\bS)(x,\xi) ,\quad
   \boldsymbol{\psi}(x,\xi) = \gM^{-1}(\boldsymbol{\varphi})(x,\xi) .
\end{equation}
Using the definition of Weyl solution, we find that
\begin{equation} \label{eq:wrFTheta}
  \mathbf{w}(x,\xi) = \mathbf{r}(x,\xi) \bF_{\Theta}(\xi) ,
\end{equation}
where we write
\begin{equation} \label{eq:rcols}
   \mathbf{r} = [r _P \,\,  r _S] .
\end{equation}
Equation (\ref{eq:JfB}) implies that
\[
  \bB(\mathbf{r}) = \ma 
      b_-(r_P) & b_-(r_S) \\ a_-(r_P) & a_-(r_S) \am
      = \mathbf{\chi}_{\operatorname{I}} = D^{\rm a}(\xi) ,
\]
where we used (\ref{eq:Bwf}).
Substituting (\ref{eq:JfB}) into (\ref{eq:wrFTheta}), we get
\begin{equation}
   \mathbf{w}(x,\xi) = \mathbf{r}(x,\xi)
                   (D^{\rm a}(\xi))^{-1} \bB(\mathbf{w}) .
\end{equation}

Now, recalling the relation between a solution to the system
(\ref{1rw(1.1)})--(\ref{1rw(1.2)}) (cf.~(\ref{eq:rcols}))
\[
   \mathbf{r}_\bullet(x,\xi)
   = \ma r_{\bullet,1}(x,\xi) \\
         r_{\bullet,2}(x,\xi) \am
\]
and a solution to the system (\ref{Rayleigh1})-(\ref{Rayleigh2})
\[
   \tilde{\mathbf{r}}_\bullet(x,\xi)
   = \ma \ii r_{\bullet,1}(-Z,\xi) \\
             r_{\bullet,2}(-Z,\xi) \am ,
\]
where $\bullet$ stands for either $P$ or $S$, the Neumann-to-Dirichlet
map for the Rayleigh problem is given by
\begin{align}
   \bND(\xi) =& \tilde{\mathbf{r}}(0,\xi) \, (D^{\rm a}(\xi))^{-1}
                \ma 0 & \ii \\ -1 & 0 \am
   = (\tilde{\boldsymbol{\theta}}(0,\xi)
         + \tilde{\boldsymbol{\psi}}(0,\xi) \bM(\xi))
         (D^{\rm a}(\xi))^{-1} \ma 0 & \ii \\ -1 & 0 \am
\nonumber\\
   =& \left[\ma \displaystyle{
   \ii \frac{\mu_0}{\mu(0)}} & 0 \\ 0 & 0 \am
     + \ma 0 & \displaystyle{\frac12 \ii} \\ \displaystyle{
       -\frac{\mu_0}{\mu(0)} \xi} & 0 \am
       \bM(\xi)\right]
           \ma \displaystyle{\frac{1}{2 \mu_0 \xi}} & 0 \\
           \displaystyle{\frac{\mu'(0)}{\mu^2(0)} \frac{1}{\xi}} &
           \displaystyle{\frac{\ii}{\mu(0)}} \am
\label{eq:C-NDM}
\end{align}
(cf.~(\ref{eq:rtttpM})).
This equation provides a direct relationship between the Weyl matrix
and the (observable) Neumann-to-Dirichlet map and, more specifically,
between the associated spectral data as $\boldsymbol{\theta}$ and
$\boldsymbol{\psi}$ are entire functions in $\xi$. Substituting
(\ref{eq:Mhhsp}) into the equation above yields the
Neumann-to-Dirichlet map in a homogeneous half space.

\bibliography{references}

\end{document}